\documentclass[reqno,a4paper]{amsart}

\usepackage[latin1,utf8]{inputenc}
\usepackage[english]{babel}
\usepackage{eucal,amsfonts,amssymb,amsmath,amsthm,epsfig,mathrsfs}
\usepackage{cancel,soul}
\usepackage{color}
\usepackage{amscd,amsxtra}
\usepackage{enumerate}  
\usepackage{latexsym}
\usepackage{bm}
\usepackage{hyperref}

\allowdisplaybreaks

\newtheorem{thm}{Theorem}[section]
\newtheorem{hyp}[thm]{Hypotheses}{\rm}
\newtheorem{lemm}[thm]{Lemma}

\newtheorem{prop}[thm]{Proposition}
\newtheorem{defi}[thm]{Definition}
\newtheorem{rmk}[thm]{Remark}{\rm} 
\newtheorem{ex}{Example}[section]

\newcommand{\R}{{\mathbb R}}
\newcommand{\N}{{\mathbb N}}
\newcommand{\E}{{\mathbb E}}
\newcommand{\eps}{\varepsilon}
\newcommand{\Id}{{\operatorname{I}}}
\newcommand{\abs}[1]{{\left|#1\right|}}
\newcommand{\norm}[1]{{\left\|#1\right\|}}
\newcommand{\scal}[2]{{\left\langle #1,#2\right\rangle}}
\newcommand{\dscal}[3]{{_{#3}\left\langle #1,#2\right\rangle}_{{#3}^*}}
\newcommand{\eqsys}[1]{{\left\{\begin{array}{ll}#1\end{array}\right.}}

\begin{document}

\frenchspacing

\title[]{Stochastic dissipative systems in Banach spaces driven by L\'evy noise}

\author[]{{Davide A. Bignamini \& Enrico Priola$^*$}}

\address[D. A. Bignamini]{Dipartimento di Scienza e Alta Tecnologia (DISAT), Universit\`a degli Studi dell'In\-su\-bria, Via Valleggio 11, 22100, Como, Italy.}
\email{\textcolor[rgb]{0.00,0.00,0.84}{da.bignamini@uninsubria.it}}

\address[E. Priola]{Dipartimento di Matematica, Università degli studi di Pavia, Via Adolfo Ferrata 5, 27100, Pavia, Italy.}
\email{\textcolor[rgb]{0.00,0.00,0.84}{enrico.priola@unipv.it}}

\subjclass[2020]{
60H15, 60G51, 35R60
}

\keywords{Generalized mild solution,  stochastic reaction-diffusion equations, dissipative systems, L\'evy noise}

\thanks{$^*$Corresponding author.}
\date{\today} 
 
\begin{abstract}
In this paper, we are interested in the  well-posedness of stochastic reaction diffusion equations like
\begin{gather} \label{ciaoI}
\eqsys{
\vspace{1pt}
dX(t)(\xi)=\big[\Delta_\xi X(t)(\xi)-p(X(t)(\xi))\big]dt+RdW(t)+dL(t) ,\quad t\in [0,T],\,\xi\in(0,1)\\
X(0)(\xi)=x(\xi)\in L^2(\mathcal{O}),\phantom{aaaaaaaaaaaaaaaaaaaaaaaaaaaaaa} \xi\in(0,1),
}
\end{gather} 
where $\mathcal{O}$ is a  bounded open domain of $\R^d$ with regular boundary, $d\in\N$, $p:\R\rightarrow\R$ is a polynomial of odd degree with positive leading coefficient, $R$ is a linear bounded operator on $L^2(\mathcal{O})$, $\{W(t)\}_{t\geq 0}$ is a $L^2(\mathcal{O})$-cylindrical Wiener process and $\{L(t)\}_{t\geq 0}$ is a pure-jump L\'evy process on $L^2(\mathcal{O})$. We complement the equation \eqref{ciaoI} with suitable boundary conditions on $\partial \mathcal{O}.$ In \cite[Chapter 10]{Pes-Zab2007} and \cite{Mar-Pre-Roc2010,Mar-Roc2010} the authors study  existence and uniqueness of  mild solutions for every $x\in L^p(\mathcal{O})$, for some suitable $p\geq 2$.
The  results of this paper allow to study reaction diffusion equations also on the space of continuous function $C(\overline{O})$. This  seems to be  new in the L\'evy case (it is already done in the Wiener case in \cite[Chapter 6]{Cer2001}).
 We also discuss and review  the previous cited works with the aim of unifying the different frameworks. We underline that, when $R=0$, for every $x\in C(\overline{O})$ (or $x\in L^p(\mathcal{O})$) the mild solution to \eqref{ciao} has a càdlàg modifications in $C(\overline{O})$ (or $\in L^p(\mathcal{O})$), even if $\{L(t)\}_{t \geq 0}$ is not a L\'evy process taking values in $C(\overline{O})$ (or $\in L^p(\mathcal{O})$).  This phenomenon for  the linear problem
(i.e.,  $F\equiv 0$ in the SPDE) has been investigated in  \cite{Pes-Zab2013};  see also the references therein. 
 \end{abstract}

\maketitle


\section{Introduction}
In this paper we study well-posedness of stochastic dissipative systems driven by L\'evy noise in infinite dimensions. Our main examples are 
stochastic reaction diffusion equations. Let us describe our setting.
 
Let $H$ be a real separable Hilbert space and let $E$ be a real separable Banach space continuously embedded in $H$.  Let $A:{\rm Dom}(A)\subseteq H\rightarrow H$ be a negative self-adjoint operator, in particular $A$ generates an analytic semigroup $\{e^{tA}\}_{t\geq 0}$ on $H$. Let $\{W(t)\}_{t\geq 0}$ be a  $H$-cylindrical Wiener process defined on a complete filtered probability space $(\Omega,\mathcal{F},\{\mathcal{F}_t\}_{t\in [0,T]},\mathbb{P})$. Let $F:E\subseteq{H}{\rightarrow}{H}$ be a (smooth enough) function. Let us first consider the SPDE
\begin{gather}\label{eqW}
\eqsys{
\vspace{1pt}
dX(t)=\big[AX(t)+F(X(t))\big]dt+RdW(t), & t\in [0,T];\\
X(0)=x\in {H}.
}
\end{gather} 
When the nonlinear part $F$ of the drift satisfies suitable monotonicity assumptions, SPDEs of the form \eqref{eqW} are sometimes referred to as dissipative stochastic systems ($R:H\rightarrow H$ is a  bounded linear operator). There exists a vast literature where the existence and uniqueness of a generalized mild solution to \eqref{eqW}  are studied, see for instance  \cite{Big2021,Cer2001,Dap2004,Dap-Zab2014}. We refer to \cite{Agr-Ver2025} for a general overview about equation of the type \eqref{eqW}. In this paper we study more general dissipative stochastic systems
\begin{gather}\label{eqFOint}
\eqsys{
\vspace{1pt}
dX(t)=\big[AX(t)+F(X(t))\big]dt+RdW(t)+dL(t), & t\in [0,T];\\
X(0)=x\in {H},
}
\end{gather}
where $L=\{L(t)\}_{t\geq 0}$ is a pure-jump Lévy process on $H$ (clearly, we can also treat the cylindrical Wiener-case corresponding to $L=0$ and $R=\Id_H$). A significant example covered by the assumptions of this paper is a class of stochastic reaction-diffusion equations of the type
\begin{gather}\label{rec} 
\eqsys{
\vspace{1pt}
dX(t)(\xi)=\big[\Delta_\xi X(t)(\xi)-p(X(t)(\xi))\big]dt+dW(t)+dL(t) , & t\in [0,T];\;\xi\in\mathcal{O}\\
\vspace{1pt}
\mathcal{B}X(t)(\xi)=0,\quad & t>0,\; \xi\in\partial\mathcal{O},\\
X(0)=x\in L^2(\mathcal{O}),\quad & \xi\in\overline{\mathcal{O}}
}
\end{gather}
where $\mathcal{O}$ is a open bounded subset of $\R^d$ with $d\in\N$, having regular boundary, $p:\R\rightarrow\R$ is a polynomial of odd degree with positive leading coefficient and $\mathcal{B}$ is an operator acting on the boundary of $\mathcal{O}$ representing the boundary conditions of the Laplacian operator.

Some existing works in the literature study mild solutions to dissipative stochastic systems like \eqref{eqFOint}, see \cite[Chapter 10]{Pes-Zab2007} and \cite{Mar-Pre-Roc2010,Mar-Roc2010}. We discuss and  review the previous papers about the existence and uniqueness of the mild and generalized mild solution to \eqref{eqFOint} within a general framework that includes known results and extends the theory to new cases. In particular the abstract result in this paper extends the framework of \cite{Big2021} and \cite[Chapter 6]{Cer2001} to the Lévy case that allows to study \eqref{rec} on the space of real continuous functions $C(\overline{\mathcal{O}})$. Equation \eqref{rec} can also be studied on $L^p(\mathcal{O})$-spaces with respect to the Lebesgue measure, for some suitable $p \geq 2$; however, working on $C(\overline{\mathcal{O}})$ allows us to extend the theory to drift $F$ which  are well defined only on $C(\overline{\mathcal{O}})$; see, for instance,  Example \ref{ExNuovo}.
  
We underline an interesting consequence of the results in this paper. Under suitable assumptions on the coefficients of \eqref{rec}, we will prove that for every $x \in C(\overline{\mathcal{O}})$ (or $x \in L^p(\mathcal{O})$, respectively), the mild solution to \eqref{rec} has a càdlàg modifications in $C(\overline{\mathcal{O}})$ (or $L^p(\mathcal{O})$, respectively), even if $\{L(t)\}_{t \geq 0}$ is not a Lévy process in $C(\overline{\mathcal{O}})$ (or $L^p(\mathcal{O})$, respectively).   See Remark \ref{remark-finale} for further details. We emphasize that the results on the regularity of the trajectories of the solution to the linear problem (i.e., with $F\equiv 0$ in \eqref{eqFOint}) established in \cite{Pes-Zab2013} are important for the approach used in this paper. Accordingly, we discuss  these results in our framework; see Section \ref{Levy}. Moreover, we emphasize that, to cover both the case $E = C(\overline{\mathcal{O}})$ and the case $E=L^p(\mathcal{O})$, we need to assume two alternative assumptions on $F$, Hypotheses \ref{Hyp1C} and Hypotheses \ref{Hyp2L}, respectively. In particular, 
 Hypotheses \ref{Hyp2L} covers the specific case studied in \cite[Section 2.1]{Mar-Roc2010}; {see also Example \ref{ExLp} 
 where we check Hypotheses \ref{Hyp2L}}.\\
 In Remark \ref{Open-question}, we have collected some related open problems. 

Before delving into the detailed results of this paper, we review some works on SPDEs driven by jump processes. Such papers use different approaches and may consider coefficients which are not necessarily of monotone type; see also the references within these papers.
 
We mention the paper \cite{CH1}, one of the seminal papers about infinite-dimensional Ornstein-Uhlenbeck processes with jumps. 
The basic theory and several applications of
SPDEs with L\'evy noise are considered in  the monograph \cite{Pes-Zab2007}.
 Existence and uniqueness results  for equations with Lipschitz type coefficients are proved in \cite{Bal2014,Cho2017,Cho-Dal-Hum2017, Pri-Shi-Xu-Zab2012} under various assumptions.  Non-degenerate SPDEs with non-Lipschitz  coefficients are considered in  \cite{Xio-Yan2019, Yan-Zho2017}. 
 We refer to \cite{Kos-Rie2022} which investigates  existence and uniqueness of  mild solutions to  SPDEs  driven by an infinite dimensional multiplicative $\alpha$-stable noise. 

In the random field approach, SPDEs driven by multiplicative $\alpha$-stable noise were studied in \cite{Mue1998} for $\alpha\leq 1$, and in \cite{Myt2002} for $1< \alpha<2$.
 
In the lecture notes \cite{Deb-Hog-Imk2013} the dynamics of the Stochastic Chafee--Infante equation with small Lévy noise is studied.
 We refer to \cite{Brz-Liu-Zhu2014} where it is studied the existence of variational strong solutions to a very large class of SPDEs with locally monotone coefficients driven by a Lévy noise. Moreover, in the variational approach, \cite{Mar-Sca2020} studies some regularity properties with respect to the initial data for the solution of \eqref{eqFOint}. SPDEs driven by L\'evy noise of interest in fluid dynamics are considered in \cite{Birnir}, \cite{BPZZ} and the references therein.  
 
 We conclude by citing the pioneering work \cite{Gyo-Kri1980}, where stochastic equations with locally monotone coefficients driven by semimartingales in Hilbert spaces are considered. These equations featured random coefficients with finite-dimensional image but much weaker regularity assumptions than the standard ones are assumed.

Below, we proceed to explain in more detail the contents of this paper.
 
\vskip 1mm
We begin to recall the definition of a generalized mild solution to \eqref{eqFOint}.
If $E={H}$ and $F$ is smooth enough (for example Lipschitz continuous) then, for every $x\in{H}$, it is possible to prove that there exists a unique $\{\mathcal{F}_t\}_{t\geq 0}$-adapted process $\{X(t,x)\}_{t\in [0,T]}$ with càdlàg trajectories in $H$ $\mathbb{P}$-a.s  that solves the mild form of \eqref{eqFOint}, namely for every $t\in [0,T],$
\begin{align}\label{mildform}
X(t)=e^{tA}x+\int_0^te^{(t-s)A}F(X(s))ds+W_A(t)+L_A(t),\quad \mathbb{P}{\rm-a.s.},
\end{align}
where 
\begin{equation}\label{Con-Con}
W_A(t):=\int_0^te^{(t-s)A}RdW(s),\qquad L_A(t):=\int_0^t e^{(t-s)A}dL(t),
\end{equation}
see Section \ref{WP-SPDE} for the complete definition.
 However, if $E\subset{H}$, then \eqref{mildform} may not make sense for every $x\in{H}$, since it is not guaranteed that there exists a process $\{X(t,x)\}_{t\in [0,T]}$ such that verifies \eqref{mildform} and its trajectories live in $E$. So, we require a more general notion of a solution: the generalized mild solution, see for instance \cite{Cer2001,Dap-Zab2014} in the cylindrical Wiener-case when $L=0$ and the references therein. 
 
 The approach to construct a generalized mild solution involves assuming that the operator $A$ and the function $F$ possess certain ``good" properties on $E$ to guarantee that, for every $x\in E$, the SPDE \eqref{eqFOint} has a unique mild solution $\{X(t,x)\}_{t\in [0,T]}$ with càdlàg trajectories in $E$. Finally, by exploiting the density of $E$, it is proved that for every $x\in {H}$, there exists a process $\{X(t,x)\}_{t\in [0,T]}$ such that
\begin{equation}\label{mildgen}
\lim_{n{\rightarrow}\infty}\sup_{t\in [0,T]}\norm{X(\cdot,x_n)-X(\cdot,x)}_{H} =0,\quad\forall T>0\quad \mathbb{P}-{\rm a.s.},
\end{equation}
where $\{x_n\}_{n\in\N} \subseteq E$ is a sequence converging to $x$ and $\{X(t,x_n)\}_{t\in [0,T]}$ is the unique mild solution to \eqref{eqFOint}, with initial datum $x_n$. We call generalized mild solution of \eqref{eqFOint} the process $\{X(t,x)\}_{t\in [0,T]}$ that satisfies \eqref{mildgen}, see Theorem \ref{limmild}.
 
In Section \ref{WP-SPDE}, to perform the procedure presented above, we will first use  a  usual trick (see for instance \cite[Section 7.2.3]{Dap-Zab2014}), i.e., we will study the equation satisfied by
$$
Y(t) = X(t) - W_A(t)-L_A(t).  
$$
This allows us to work pathwise and thereby reducing the problem of the well-posedness of \eqref{eqFOint} to the well-posedness of an auxiliary  partial differential equation (PDE) with  random coefficients. The appendix of the present paper is devoted to the analysis of the well-posedness of such PDE within a framework that is even more general than the one employed for the SPDE. When $L = \{L(t)\}_{t\geq 0}$ is not identically zero to use the previous trick, it will be crucial to assume that the stochastic convolution process $\{L_A(t)\}_{t\geq 0}$ has càdlàg paths in $E$, $\mathbb{P}$-a.s.

Note that in the pure-jump Lévy case, studying the regularity of trajectories of stochastic convolution processes is a quite complex issue. For instance, some counterexamples have proved that their trajectories may not even be càdlàg in the space where the cylindrical L\'evy process is defined, see for instance \cite{Brz-Gol-Imk-Pes-Pri-Zab2010, Liu-Zha2012, Pes-Zab2013} (this is in contrast with the cylindrical Wiener-case when $L=0$; see Remark \ref{BvsL} for a detailed discussion). {On this respect in Section \ref{Levy}, we review a result from \cite{Pes-Zab2013} within our context. This provides a useful criterion to establish the required càdlàg properties of $\{L_A(t)\}_{t\geq 0}$ appearing in \eqref{Con-Con}; see Proposition \ref{PZvariante} and Remark \ref{ErZabPes}. }

%

\subsection{Notation}
Let ${\mathcal{K}}$ and ${\mathcal{Z}}$ be two separable Banach spaces. We denote by ${\mathcal{K}^*}$ the dual of $\mathcal{K}$. We denote by $\mathcal{L}(\mathcal{K})$ the space of linear bounded operators from ${\mathcal{K}}$ into itself. We denote by $\Id_{\mathcal{K}}$ the identity operator on ${\mathcal{K}}$. Let $H$ be a separable Hilbert space and let $B:{\rm Dom}(B)\subseteq H\rightarrow H$ be a linear operator, we say that $B$ is self-adjoint if $B=B^*$ and ${\rm Dom}(B)$ is dense in $\mathcal{K}$, further it is also negative if for every $x\in H\backslash\{0\}$ we have
\begin{equation}\label{negativo}  
\scal{Bx}{x}_H<0.
\end{equation}  
Let $I$ be an interval of $\R$ we denote by $C_b(I;{\mathcal{K}})$ the set of the continuous bounded functions from $I$ to $\mathcal{K}$, if $I$ is compact we simply write $C(I;{\mathcal{K}})$.
We say that a function $f:I\rightarrow{\mathcal{K}}$ is càdlàg if it is right-continuous with finite left limit at every $t\in I$. Conversely we say that $f$ is càglàd if it is left-continuous with finite right limit at every $t\in I$; in the extremes of $I$ we suitably adapt the definition. We note that if $f:\R\rightarrow{\mathcal{K}}$ is càdlàg then 
\[
g(t):=f(t^-)=\lim_{s\rightarrow t^-}f(s),\quad t\in\R,
\]
is càglàd. { We recall that any càdlàg or càglàd function has at most a countable number of discontinuities.}\\
If $\mathcal{Z}\subseteq \mathcal{K}$ and  $G: {\rm Dom}(G)\subseteq {\mathcal{K}}{\rightarrow} {\mathcal{K}}$ then we call part of $G$ in ${\mathcal{Z}}$ the function $G_{{\mathcal{Z}}}:{\rm Dom}(G_{{\mathcal{Z}}})\subseteq {\mathcal{Z}}{\rightarrow} {\mathcal{Z}}$ defined by
\begin{align*}
{\rm Dom}(G_{{\mathcal{Z}}})&:=\{x\in {\rm Dom}(G)\cap {\mathcal{Z}}\; :\; G(x)\in {\mathcal{Z}} \},\\& G_{{\mathcal{Z}}}(x):=G(x),\; \forall\;x\in {\rm Dom}(G_{{\mathcal{Z}}}).
\end{align*}

The stochastic dissipative systems are the main focus of this paper, so, due to its relevance, we recall the definition of dissipative maps.

\begin{defi}\label{defi-dissi}
A map $f:{\rm Dom}(f)\subseteq \mathcal{K}{\rightarrow}\mathcal{K}$ is said to be dissipative if, for every $\alpha>0$ and $x,y\in {\rm Dom}(f)$, we have
\begin{equation}\label{disban1}
\norm{x-y-\alpha(f(x)-f(y))}_{\mathcal{K}}\geq \norm{x-y}_{\mathcal{K}}.
\end{equation}
If $f=A$ with $A\in\mathcal{L}(\mathcal{K})$, then \eqref{disban1} reads as
\[
\norm{(\lambda\Id-A)x}_{\mathcal{K}}\geq \lambda \norm{x}_{\mathcal{K}}, \quad \forall \lambda>0,\; x\in{\rm Dom}(A)
\]
We say that $f$ is m-dissipative if the range of $\lambda\Id-f$ is the whole space $\mathcal{K}$ for some $\lambda>0$ (and so for all $\lambda>0$).\\
If $\mathcal{K}$ is a Hilbert space $f$ is dissipative if and only if for every $x,y\in {\rm Dom}(f)$, we have
\begin{equation}\label{dissiHilbert}
\scal{f(x)-f(y)}{x-y}_{\mathcal{K}}\leq 0.
\end{equation}
\end{defi}
In Proposition \ref{dissiBanach} in the Appendix, by exploiting the notion of subdifferential, we provide a useful characterization of dissipative maps in Banach spaces, analogous to \eqref{dissiHilbert}.

Let $(\Omega,\mathcal{F},\mathbb{P})$ be a complete probability space. Let $\xi:(\Omega,\mathcal{F},\mathbb{P}){\rightarrow} ({\mathcal{K}},\mathcal{B}({\mathcal{K}}))$ be a random variable, we denote by $\mathbb{E}[\xi]$ the expectation of $\xi$ with respect to $\mathbb{P}$.\\
Let $\{Y(t)\}_{t\in I}$ be a ${\mathcal{K}}$-valued stochastic process. We say that $\{Y(t)\}_{t\in I}$ is càdlàg if the map $Y(\cdot):I {\rightarrow} {\mathcal{K}}$ is càdlàg $\mathbb{P}$-a.s.

\section{The stochastic convolution}\label{Levy}

This section is devoted to the regularity of the two stochastic convolution processes defined in \eqref{Con-Con}. Let $H$ be a real separable Hilbert space. Let $A:{\rm Dom}(A)\subseteq H\rightarrow H$ be a negative and self-adjoint operator (see \eqref{negativo}). In particular $A$ generates a strongly continuous analytic semigroup $\{e^{tA}\}_{t\geq 0}$. For every $\rho\geq 0$, we set $H_\rho:={\rm Dom}((-A)^\rho)$ and we recall that is a separable Hilbert equipped with the following norm 
\begin{align*}
\abs{x}_\rho=\abs{(-A)^{\rho}x}_H,\qquad x\in H.
\end{align*}
Moreover, by the theory of analytic semigroups (see \cite[Chapter 2]{Lun1995}), for every $0\leq\rho_1\leq\rho_2$ there exists $C_{\rho_1,\rho_2}>0$ such that 
\begin{equation}\label{analitic3}
\norm{e^{tA}x}_{\rho_2}\leq C_{\rho_1,\rho_2}t^{-(\rho_2-\rho_1)}\norm{x}_{\rho_1},\qquad x\in H_{\rho_1},\; t>0
\end{equation}
Moreover for every $0\leq\eta_1<\eta_2$ we have that
\begin{equation}\label{ContEmb}
H_{\eta_2}\subseteq H_{\eta_1}\,\, {\rm (continuously)},
\end{equation}
see \cite[Chapter 4]{Lun1999} for further details about powers of positive operators.
We consider a linear bounded operator $R:H\rightarrow H$ and a $H$-cylindrical Wiener process $\{W(t)\}_{t\geq 0}$ on $H$ defined on a complete filtered probability space $(\Omega,\mathcal{F},\{\mathcal{F}_t\}_{t\geq 0},\mathbb{P})$, namely a process formally defined by
\begin{equation}\label{Wcilindrico}
W(t):=\sum_{k=1}^{+\infty}\beta_k(t)e_k,\qquad t>0,
\end{equation}  
where $\beta_1,\ldots,\beta_k,\ldots$ are real standard independent Brownian motion and $\{e_k\}_{k\in\N}$ is a orthonormal basis of $H$.  We consider the stochastic convolution process $\{W_A(t)\}_{t\geq 0}$ defined by
\begin{equation*}
W_A(t):=\int_0^te^{(t-s)A}RdW(s),\qquad t\geq 0.
\end{equation*}
The theory about $\{W_A(t)\}_{t\geq 0}$ is quite classical in literature, see for instance \cite[Chapters 4-5]{Dap-Zab2014} and \cite[Chapter 2]{Liu-Roc2015}. Below, we recall a result that provides a useful criterion for verifying the regularity  properties on $\{W_A(t)\}_{t\geq 0}$, which will be important in the sequel.
 
\begin{prop}\label{PropCarlo}
Let $T>0$, $d\in\N$, $H=L^2([0,1]^d)$, let $A$ be the realization of the Laplacian operator in $H$ with Dirichlet boundary condition and let $R=(-A)^\delta$ with $\delta\in\R$. Then the following assertions hold true. 
\begin{enumerate}
\item If $\delta=0$ then $\{W_{A}(t)\}_{t\geq 0}$ is a continuous process (with values) in $H_\gamma$ with $\gamma<1/4$.
\item If $\delta > (d-2)/4$ then $\{W_{A}(t)\}_{t\geq 0}$ is a continuous process in $C([0,1]^d)$. Moreover for every $p\geq 2$
\[
\mathbb{E}\left[\sup_{t\in [0,T]}\norm{W_A(t)}_{C([0,1]^d)}^p\right]<+\infty.
\]
\end{enumerate}
\end{prop}
We refer to \cite[Section 4.2]{Fuh-Orr2015} for a proof of Proposition \ref{PropCarlo}. We refer to \cite[Lemma 6.1.2]{Cer2001} for a result analogous to Proposition \ref{PropCarlo} but in a more general framework.

Unfortunately, in the general Lévy case, the regularity theory for the stochastic convolution is much  more involved than the one 
in the Wiener case. 
  
First of all we recall that a Lévy process in $H$ is a $H$-valued stochastic process $\{L(t)\}_{t\geq 0}$ defined on a complete filtered probability space $(\Omega,\mathcal{F}, \{\mathcal{F}_t\}_{t\geq 0},\mathbb{P})$, adapted, continuous in probability, having stationary independent increments, càdlàg trajectories, and such that $L(0)=0$, $\mathbb{P}$-a.s. As in the finite dimensional case, the Lévy--It\^o decomposition of  $\{L(t)\}_{t\geq 0}$ holds true (see \cite[Theorem 6.8]{Pes-Zab2007}).  


In this paper we consider a pure-jump Lévy process  $\{L(t)\}_{t\geq 0}$ in $H$ having the following Lévy--It\^o decomposition
\begin{align}\label{L_K}
L(t)=mt+\int_0^t\int_{\{z\in H\; :\; \abs{z}_H> 1\}}z\pi(ds,dz)+\int_0^t\int_{\{z\in H \;:\; \abs{z}_H\leq 1\}}z\widehat{\pi}(ds,dz),\quad t>0,
\end{align}
where $m\in H$, $\pi$ is the Poisson random measure of $\{L(t)\}_{t\geq 0}$, $\widehat{\pi}$ is the compensated Poisson random measure
\[
\widehat{\pi}(ds,dz)=\pi(ds,dz)-\nu(dz)ds,  
\]
where  $\nu$ is the Lévy measure associated with $\{L(t)\}_{t\geq 0}$.
 We recall that  for every  $t > 0$ and $A\in\mathcal{B}\left(H\setminus\{0\}\right)$
\begin{align*}
&\pi((0,t] \times A)(\omega):= \sharp\left\{0 < s\leq t\ : \ [L(s)-L(s^-)](\omega)\in A\right\}, \\
&\nu(A):=\mathbb{E}\left[\pi((0,1] \times A)\right],  
\end{align*}
where $\omega\in\Omega_0$ and $\Omega_0 \in {\mathcal F}$ such that $\mathbb{P}(\Omega_0)=1$ and for every $\omega\in\Omega_0$ the map $t\rightarrow L(t)(\omega)$ is càdlàg. See \cite[Chapter 6]{Pes-Zab2007} for further details about Poisson random measures. For simplicity, we assume 
$$
m=0;
$$ 
however, all results in this section can be easily extended to the case $m\neq 0$ with suitable assumptions on $m$.

\vskip 1mm   
We mention that the notion of cylindrical Lévy processes is more delicate than that of cylindrical Wiener processes. We refer to \cite{App-Rie2009, Jak-Rie2017} for the general study of
cylindrical Lévy processes (see also \cite{Brz-Zab2010, Pri-Zab2011,  Brz-Gol-Imk-Pes-Pri-Zab2010,Pes-Zab2013} and the references therein). 

We start to investigate the regularity of the trajectories of the stochastic convolution process $\{L_A(t)\}_{t\geq 0}$ defined by
\[
L_A(t)=\int_0^t e^{(t-s)A}dL(s),\qquad t\geq 0,\; \mathbb{P}-{\rm a.s.}
\]
(cf. \cite{CH1} for the precise definition of the previous stochastic integral).

The problem of existence of a càdlàg version for $L_A(t)$ is also delicate. 
 In contrast to the cylindrical Wiener case in the case of cylindrical pure-jump L\`evy noise on $H$ there are counterexamples to existence of a càdlàg version  for 
$L_A(t)$ in $H$;  see for instance \cite{Brz-Gol-Imk-Pes-Pri-Zab2010, Liu-Zha2012, Pes-Zab2013}. 

On the other hand if $\{L(t)\}_{t\geq 0}$ is a L\'evy process in $H$ then  we have existence of such version as the following remark shows.

\begin{rmk}\label{RmkMPR}
If $\{L(t)\}_{t\geq 0}$ is a Lévy process on $H$, 
then, proceeding as in  \cite[Section 9.4]{Pes-Zab2007}, it is possible to prove that the process $\{L_A(t)\}_{t\geq 0}$ is càdlàg in $H$.  For this purpose, one can also exploit maximal inequalities; see for instance \cite{Mar-Pre-Roc2010,Mar-Roc2010}.
\end{rmk}

Let $\delta\geq 0$. We note that if $\{L(t)\}_{t\geq 0}$ is a Lévy process on $H_{\delta}$ then, by definition of $\nu$, we have
\[
\nu (H \setminus H_{\delta}) =0.
\]
In general, the converse is not true; however, in the next proposition we will prove a sufficient condition to guarantee that $\{L(t)\}_{t\geq 0}$ is a Lévy process on $H_{\delta}$.
\begin{prop}[Proposition 2.6 of \cite{Pes-Zab2013}]\label{valoriHgamma}
Let $\delta\geq 0$. Assume that $\nu(H\setminus H_\delta)=0$, $m =0$,   and  
\begin{equation}\label{delta-cond}
\int_H\abs{z}_\delta^2\nu(dz) = \int_{H_{\delta}}\abs{z}_\delta^2\nu(dz) <+\infty. 
\end{equation}
Then the following statements hold.
\begin{enumerate}
\item $\{L(t)\}_{t\geq 0}$ is a Lévy process on $H_\delta$.
\item $\{L_A(t)\}$ is a mean square continuous $H_{\gamma}$-valued stochastic process, for every $\gamma<1/2+\delta$,  namely for every $t_0\geq 0$ we have
\begin{equation}\label{Mean-cont}
\lim_{t\rightarrow t_0}\mathbb{E}\left[\norm{L_A(t)-L_A(t_0)}_\gamma^2\right]=0,
\end{equation}
in the case $t_0=0$ we consider only the right limit in \eqref{Mean-cont}.
\end{enumerate}
\end{prop}
\begin{proof}
The statement in (1) is a direct consequence of \eqref{delta-cond} and \cite[Proposition 2.6(i)]{Pes-Zab2013} (with $H$ of \cite{Pes-Zab2013} which is our $H_\delta$). 
We also have 
\begin{gather} \label{ciao}L(t)= 
\int_0^t\int_{\{z\in H_{\delta}\; :\; \abs{z}_{\delta}> 1\}}z\pi(ds,dz)+\int_0^t\int_{\{z\in H_{\delta} \;:\; \abs{z}_{\delta}\leq 1\}}z\widehat{\pi}(ds,dz)
\end{gather}
We prove the statement in (2) only in the case $\delta<\gamma<1/2+\delta$, the case $\gamma\le \delta$ follows from \eqref{ContEmb}. We want apply \cite[Proposition 2.6(iii)]{Pes-Zab2013} (with $U=H_\delta$ and $H=H_{\gamma}$). First of all we note that
\[
\int_{\abs{z}_\delta>1}\abs{z}_\delta\nu(dz)<\int_{H_\delta}\abs{z}_\delta^2\nu(dz)<+\infty.
\]
 We know that
$a=\int_{\abs{z}_\delta >1}z \nu(dz) \in H_\delta$. Moreover, for every $t> 0$, by \eqref{analitic3} (with $\rho_2=\gamma$ and $\rho_1=\delta$)
\[
\int_0^t\abs{e^{sA}a}_\gamma ds\leq \abs{a}_\delta \int_0^t\frac{C}{s^{\gamma-\delta}}ds<+\infty,
\]
hence $\int_0^te^{sA}ads\in H_\gamma$. In the same way, we have
\[
\int_0^t\int_{H_\delta}\abs{e^{sA}z}^2_{\gamma}\nu(dz)ds\leq C_{\delta,\gamma}\int_{H_\delta}\abs{z}_\delta^2\nu(dz) \int_0^t\frac{C}{s^{2(\gamma-\delta)}}ds<+\infty,
\]
hence applying \cite[Proposition 2.6(iii)]{Pes-Zab2013} (with $U=H_\delta$ and $H=H_{\gamma}$) we obtain the statement of point 2.
\end{proof}

A first result in the same spirit of Proposition \ref{PropCarlo} is contained in \cite{Pes-Zab2013}, so below we are going to review such result in our context. Hence, in the next proposition, we revisit \cite[Theorem 3.1]{Pes-Zab2013} in our setting. We underline that we found some points in the proof of \cite[Theorem 3.1]{Pes-Zab2013} that are not completely clear to us. Consequently, we decided to prove it within our framework.  
Note that our result is weaker that the corresponding result stated in \cite{Pes-Zab2013}. Indeed in  Remark \ref{ErZabPes}, after the proof, we listed some unclear points about the proof in \cite{Pes-Zab2013}.
\\
We stress again that we are showing a regularity improvement for $\{L_A(t)\}_{t\geq 0}$, indeed, the process $\{L(t)\}_{t\geq 0}$ might not be càdlàg in the same space where $\{L_A(t)\}_{t\geq 0}$ is càdlàg; see for instance Example \ref{ExPZ}.

\begin{prop}\label{PZvariante}
Let $\delta\geq 0$ and $0\leq \eps\leq 1/4$. Assume that $\nu(H\backslash H_{\eps+\delta})=0$, $m =0$, and 
\begin{equation}\label{eps-int} 
\int_H\abs{z}_\delta^2+\abs{z}_{\eps+\delta}^4\nu(dz)
= \int_{H_{\epsilon + \delta}}\abs{z}_\delta^2+\abs{z}_{\eps+\delta}^4\nu(dz)
<+\infty.
\end{equation}
 Then, for every $\gamma<\eps+\delta$, the process $\{L_A(t)\}_{t\geq 0}$ is a càdlàg process in $H_{\gamma}$ and 
\[
\mathbb{E}\left[\sup_{t\in [0,T]}\abs{L_A(t)}_{\gamma}^q\right]<\infty,\qquad T>0,\; q\in [1,4).
\]
\end{prop} 
\begin{proof}
Fix $\eps\in [0,1/4]$. We focus on the case $\gamma\in (\delta, \eps+\delta)$, the case $\gamma\leq\delta$ follows from \eqref{ContEmb}.

First of all we note that, by Proposition \ref{valoriHgamma} and \eqref{eps-int}, for every $t>0$ we have 
\[
L_A(t)=\int_0^t\int_H e^{(t-s)A}dL(t)= \int_0^te^{(t-s)A}ads+\int_0^t\int_He^{(t-s)A}z\widehat{\pi}(ds,dz),
\]
where 
\[
a=\int_{\{z\in H\,:\, \abs{z}_\delta>1\}}z\nu(dz)=\int_{\{z\in H_{\delta}\,:\, \abs{z}_\delta>1\}}z\nu(dz \in H_\delta. 
\]
Since $\{e^{tA}\}_{t\geq 0}$ is an analytic semigroup then $t\in [0,+\infty)\rightarrow \int_0^te^{(t-s)A}ads\in H_\gamma$ is a continuous map for every $\gamma<\eps+\delta$, so in the sequel we assume $a=0$ for simplicity.

To prove the statement of this proposition, we will exploit a criterion derived by the Gikhman--Skorokhod test \cite[Section III.4]{Gik-Sko1990} and by the Bezandry-Fernique test \cite{Bez-Fer1990}; see also \cite[Corollary 2.2]{Pes-Zab2013} (with $p=2$). We underline that the Gikhman–Skorokhod test \cite[Section III.4]{Gik-Sko1990} requires that the process $\{L_A(t)\}_{t\geq 0}$ is separable, whereas to apply the Bezandry-Fernique test \cite{Bez-Fer1990} one requires the process is stochastically continuous which is guaranteed by Proposition \ref{valoriHgamma}(2).

We proceed as in \cite{Pes-Zab2013}.  
By the tests mentioned above if there exist $\eta,R>0$ such that for any $0<h<t<T$ we have
\begin{align}\label{stimaF}
\E\left[\abs{L_A(t+h)-L_A(t)}^2_\gamma\abs{L_A(t)-L_A(t-h)}^2_\gamma\right]\leq Rh^{1+\eta},
\end{align}
then the process $\{L_A(t)\}_{t\in [0,T]}$ has a càdlàg modification in $H_\gamma$. We are going to check \eqref{stimaF}.
By standard calculations we obtain
\begin{align*}
\E\left[\abs{L_A(t+h)-L_A(t)}^2_\gamma\abs{L_A(t)-L_A(t-h)}^2_\gamma\right]&= \E\left[\abs{B_1+B_2+B_3}^2_\gamma\abs{C_1+C_2}^2_\gamma\right]
\end{align*}
where
\begin{align*}
B_1:&=\int_0^{t-h}\int_{H}(e^{(t+h-s)A}-e^{(t-s)A})z\widehat{\pi}(ds,dz),\\
B_2:&=\int_{t-h}^{t}\int_{H}(e^{(t+h-s)A}-e^{(t-s)A})z\widehat{\pi}(ds,dz),\\
B_3:&=\int_t^{t+h}\int_{H} e^{(t+h-s)A}z\widehat{\pi}(ds,dz),\\
C_1:&=\int_0^{t-h}\int_{H}  (e^{(t-s)A}-e^{(t-h-s)A})z\widehat{\pi}(ds,dz),\\
C_2:&=\int_{t-h}^{t}\int_{H} e^{(t-s)A}z\widehat{\pi}(ds,dz).
\end{align*}
By further calculations we get
\begin{align}\label{StimaFF}
\E\big[\abs{L_A(t+h)-L_A(t)}^2_\gamma &\abs{L_A(t)-L_A(t-h)}^2_\gamma \big]\leq 27\Bigg( \E\left[\abs{B_1}^2_\gamma\abs{C_1}^2_\gamma\right]+\E\left[\abs{B_2}^2_\gamma\abs{C_2}^2_\gamma\right]\notag\\
&+\E\left[\abs{B_1}^2_\gamma\right]\E\left[\abs{C_2}^2_\gamma\right]+\E\left[\abs{B_2}^2_\gamma\right]\E\left[\abs{C_1}^2_\gamma\right]\notag\\
&+\E\left[\abs{B_3}^2_\gamma\right]\E\left[\abs{C_1}^2_\gamma\right]+\E\left[\abs{B_3}^2_\gamma\right]\E\left[\abs{C_2}^2_\gamma\right]\Bigg).
\end{align}
We make an observation that is fundamental to the next estimates. By standard properties of the analytic semigroups and by \eqref{analitic3} (with $\rho_1=0$ and $\rho_2=1-\kappa$) for every $\kappa>0$ there exists $C_{\kappa}>0$ such that for every $h>0$ and $u\in H_{\gamma+\kappa}$ we have
\begin{align}
\abs{(e^{hA}-\Id)u}_\gamma&\leq \int^h_0\abs{Ae^{sA}u}_\gamma ds= \int^h_0\abs{(-A)^{1+\gamma}e^{sA}u}_{H} ds\notag\\
&=\int^h_0\abs{(-A)^{1-\kappa}e^{sA}(-A)^{\gamma+\kappa}u}_{H} ds\leq C_{\kappa}h^{\kappa}\abs{u}_{\kappa+\gamma}\label{E1}.
\end{align}

We begin to estimate $\E[\abs{B_1}^2_\gamma]$. Let $\kappa>0$. By \cite[Theorem 8.23]{Pes-Zab2007}, \eqref{analitic3} and \eqref{E1} we have
\begin{align*}
\E\left[\abs{B_1}^2_\gamma\right]&=\int_0^{t-h}\int_{H}\abs{e^{(t+h-s)A}-e^{(t-s)A}z}^2_\gamma\nu(dz)ds\\
&=\int_0^{t-h}\int_{H}\abs{(e^{hA}-\Id)e^{(t-s)A}z}^2_\gamma\nu(dz)ds\\
&\leq C^2_{\kappa}h^{2\kappa}\int_0^{t-h}\int_{H}\abs{e^{(t-s)A}z}^2_{\kappa+\gamma}\nu(dz)ds\\
&\leq C^2_{\kappa}C^2_{\delta,\kappa+\gamma}h^{2\kappa}\int_0^{t-h}\frac{1}{(t-s)^{2\gamma+2\kappa-2\delta}}ds\int_{H}\abs{z}^2_\delta\nu(dz),\\
&\leq C^2_{\kappa}C^2_{\delta,\kappa+\gamma}h^{2\kappa}\frac{1}{1+2\delta-2\kappa-2\gamma}\left(t^{1+2\delta-2\gamma-2\kappa}-h^{1+2\delta-2\gamma-2\kappa}\right)\int_{H}\abs{z}^2_{\delta}\nu(dz)\\
&\leq   C^2_{\kappa}C^2_{\delta,\kappa+\gamma}h^{2\kappa}\frac{1}{1+2\delta-2\gamma-2\kappa}T^{1+2\delta-2\kappa-2\gamma}\int_{H}\abs{z}^2_{\delta}\nu(dz).
\end{align*}
Hence there exists $K_{\delta,\kappa,\gamma,T}>0$ such that for any $0\leq h\leq t\leq T$ we have 
\begin{align}\label{stima1}
\E\left[\abs{B_1}^2_\gamma\right]\leq K_{\delta,\kappa,\gamma,T}h^{2\kappa}\int_{H}\abs{z}^2_{\delta}\nu(dz),
\end{align}
where $\kappa>0$.

We estimate $\E[\abs{B_2}^2_\gamma]$. Let $\theta>0$ such that $\theta+\gamma-\delta<1/2$. By the same arguments used for $\E[\abs{B_1}^2_\gamma]$ (with $\kappa$ replaced by $\theta$) for any $0\leq h\leq t\leq T$ we obtain 
\begin{align*}
\E\left[\abs{B_2}_\gamma^2 \right]&\leq C^2_{\theta}C^2_{\delta,\theta+\gamma}h^{2\theta}\int_{t-h}^{t}\frac{1}{(t-s)^{2\gamma+2\theta-2\delta}}ds\int_{H}\abs{z}^2_{\delta}\nu(dz).
\end{align*}
Since $\theta+\gamma-\delta<1/2$ and $\theta>0$ is arbitrary small, there exists $K_{\gamma,\delta}>0$ such that for any $0\leq h\leq t\leq T$ we have 
\begin{align}\label{stima2}
\E\left[\abs{B_2}^2_\gamma\right]\leq K_{\gamma,\delta}h^{1+2\delta-2\gamma}\int_{H}\abs{z}^2_{\delta}\nu(dz).
\end{align}
We estimate $\E[\abs{B_3}_\gamma^2]$. Since $\gamma-\delta<1/2$, by \cite[Theorem 8.23]{Pes-Zab2007} and \eqref{analitic3} we have
\begin{align*}
\E\left[\abs{B_3}^2_\gamma\right]&=\int_t^{t+h}\int_{H}\abs{e^{(t+h-s)A}z}^2_\gamma\nu(dz)ds\\
&\leq C^2_{\delta,\gamma}\int_t^{t+h}\frac{1}{(t+h-s)^{2\gamma-2\delta}}ds\int_{H}\abs{z}^2_\delta\nu(dz)ds\\
&= C^2_{\delta,\gamma}\frac{1}{1+2\delta-2\gamma}h^{1+2\delta-2\gamma}\int_{H}\abs{z}^2_\delta\nu(dz).
\end{align*}
Hence there exists $K'_{\gamma,\delta}>0$ such that for any $0\leq h\leq t\leq T$ we have 
\begin{align}\label{stima3}
\E[\abs{B_3}^2_\gamma]\leq K'_{\gamma,\delta}h^{1+2\delta-2\gamma}\int_{H}\abs{z}^2_\delta\nu(dz)ds.
\end{align}

We estimate $\E[\abs{C_1}^2_\gamma]$. Let $\sigma>0$ such that $\sigma+\gamma-\delta<1/2$. By the same arguments used for $\E[\abs{B_1}^2_\gamma]$ for any $0\leq h\leq t\leq T$ we have 
\begin{align*}
\E\left[\abs{C_1}_\gamma^2 \right]&=\int_0^{t-h}\int_{H}\abs{e^{(t-s)A}-e^{(t-h-s)A}z}^2_\gamma\nu(dz)ds\\
&=\int_0^{t-h}\int_{H}\abs{(e^{hA}-\Id)e^{(t-s-h)A}z}^2_\gamma\nu(dz)ds\\
&\leq C^2_{\sigma}h^{2\sigma}\int_0^{t-h}\int_{H}\abs{e^{(t-s-h)A}z}^2_{\sigma+\gamma}\nu(dz)ds\\
&\leq C^2_{\sigma}C^2_{\delta,\sigma+\gamma}h^{2\kappa}\int_0^{t-h}\frac{1}{(t-s-h)^{2\gamma+2\sigma-2\delta}}ds\int_{H}\abs{z}^2_{\delta}\nu(dz)\\
&\leq  C^2_{\sigma}C^2_{\delta,\sigma+\gamma}h^{2\sigma}\frac{1}{1-2\sigma-2\gamma+2\delta}(t-h)^{1+2\delta-2\gamma-2\sigma}\int_{H}\abs{z}^2_{\delta}\nu(dz)\\
&\leq  C^2_{\sigma}C^2_{\delta,\sigma+\gamma}h^{2\sigma}\frac{1}{1-2\sigma-2\gamma+2\delta}T^{1+2\delta-2\gamma-2\sigma}\int_{H}\abs{z}^2_{\delta}\nu(dz).
\end{align*} 
Hence there exists a constant $K_{\sigma,\gamma,\delta,T}>0$ such that
\begin{align}\label{stima4}
\E\left[\abs{C_1}^2_\gamma\right]\leq K_{\sigma,\gamma,\delta,T}h^{2\sigma}\int_{H}\abs{z}^2_{\delta}\nu(dz),
\end{align}
where $\sigma>0$ and $\sigma+\gamma-\delta<1/2$. 

We estimate $\E[\abs{C_2}^2_\gamma]$. By the same arguments used for $\E[\abs{B_3}^2_\gamma]$ for any $0\leq h\leq t\leq T$ we have 
\begin{align*}
\E\left[\abs{C_2}^2_\gamma\right]&=\int_{t-h}^{t}\int_{H}\abs{e^{(t-s)A}z}^2_\gamma\nu(dz)ds\leq C^2_{\delta,\gamma}\int_{t-h}^{t}\frac{1}{(t-s)^{2\gamma-2\delta}}ds\int_{H}\abs{z}^2_\delta\nu(dz)ds\\
&= C^2_{\delta,\gamma}\frac{1}{1+2\delta-2\gamma}h^{1-2\gamma+2\delta}\int_{H}\abs{z}^2_\delta\nu(dz),
\end{align*}
hence we obtain
\begin{align}\label{stima5}
\E\left[\abs{C_2}_\gamma\right]\leq K'_{\gamma,\delta}h^{1+2\delta-2\gamma}\int_{H}\abs{z}^2_\delta\nu(dz),
\end{align}
where $K'_{\gamma,\delta}$ is the same constant of \eqref{stima3}.

Now we have to estimate $\E[\abs{B_1}^2_\gamma\abs{C_1}^2_\gamma]+\E[\abs{B_2}^2_\gamma\abs{C_2}^2_\gamma]$. We start with $\E[\abs{B_1}^2_\gamma\abs{C_1}^2_\gamma]$. By \cite[Proposition 2.4]{Pes-Zab2013} we have
\begin{align*}
\E\left[\abs{B_1}^2_\gamma\abs{C_1}^2_\gamma\right]=I_1+I_2+I_3
\end{align*}
where
{\small
\begin{align*}
&I_1:=\left(\int^{t-h}_0\int_H\abs{\left(e^{(t+h-s)A}-e^{(t-s)A}\right)z}^2_\gamma ds\nu(dz)\right)\left(\int^{t-h}_0\int_H\abs{\left(e^{(t-s)A}-e^{(t-h-s)A}\right)z}^2_\gamma ds\nu(dz)\right);\\
&I_2:=2\left[\int^{t-h}_0\int_H \scal{\left(e^{(t+h-s)A}-e^{(t-s)A}\right)z}{\left(e^{(t-s)A}-e^{(t-h-s)A}\right)z}_\gamma ds\nu(dz)\right]^2;\\
&I_3:=\int^{t-h}_0\int_H\abs{\left(e^{(t+h-s)A}-e^{(t-s)A}\right)z}^2_\gamma \abs{\left(e^{(t-s)A}-e^{(t-h-s)A}\right)z}^2_\gamma ds\nu(dz).
\end{align*} 
}
We estimate $I_1$. In the same way of \eqref{stima1} and \eqref{stima4} we have
\begin{align}\label{StimaI1}
I_1\leq K_{\kappa,\gamma,\delta,T}K_{\sigma,\gamma,\delta,T}h^{2\kappa+2\sigma}\left(\int_{H}\abs{z}^2_{\delta}\nu(dz)\right)^2,
\end{align}
where $\sigma,\kappa>0$, $\sigma+\gamma-\delta<1/2$, $K_{\kappa,\gamma,\delta,T}$ and $K_{\sigma,\gamma,\delta,T}$ are the same constant in \eqref{stima1} and \eqref{stima4}.

We estimate $I_2$. Let $\sigma,\kappa>0$ be as in \eqref{StimaI1}. By \eqref{analitic3} and \eqref{E1} we have
\begin{align*}
I_2&\leq 2\left[\int^{t-h}_0\int_H \abs{\left(e^{hA}-\Id\right)e^{(t-s)A}z}_\gamma\abs{\left(e^{hA}-\Id\right)e^{(t-h-s)A}z}_\gamma ds\nu(dz)\right]^2\\
&\leq  C^2 h^{2\sigma+2\kappa}\left(\int^{t-h}_0\frac{1}{(t-s)^{\kappa+\gamma-\delta}}\frac{1}{(t-h-s)^{\sigma+\gamma-\delta}}ds\right)^2\left(\int_H\abs{z}_\delta^2 dz\right)^2
\end{align*}
where $C:=2C_\sigma C_\kappa C_{\delta,\sigma+\gamma}C_{\delta,\kappa+\gamma}$. Since $\sigma+\gamma-\delta<1/2$ by the H\"older inequality we obtain
\begin{align*}
I_2&\leq C^2h^{2\sigma+2\kappa}\abs{\int^{t-h}_0\frac{1}{(t-s)^{2(\kappa+\gamma-\delta)}}\int^{t-h}_0\frac{1}{(t-h-s)^{2(\sigma+\gamma-\delta)}}ds}\left(\int_H\abs{z}_\delta^2 dz\right)^2 \\
&\leq \frac{T^{2+4\delta-4\gamma-2\theta-2\kappa}}{(1+2\delta-2\gamma-2\kappa)(1+2\delta-2\gamma-2\theta)}C^2h^{2\sigma+2\kappa}\left(\int_H\abs{z}_\delta^2 dz\right)^2 
\end{align*}
hence there exists $K_{\kappa,\sigma,\gamma,\delta,T}>0$ such that
\begin{align}\label{StimaI2}
I_2\leq K_{\kappa,\sigma,\gamma,\delta,T} h^{2\kappa+2\sigma}\left(\int_H\abs{z}_\delta^2 dz\right),
\end{align}
where $\sigma,\kappa>0$ and $\sigma+\gamma-\delta<1/2$.

We estimate $I_3$. Let $\sigma,\kappa>0$ be as in \eqref{StimaI1}. By \eqref{analitic3} and \eqref{E1} we have 
 \begin{align*}
I_3\leq C^2h^{2\sigma+2\kappa}\int^{t-h}_0\frac{1}{(t-s)^{2(\kappa+\gamma-\delta)}}\frac{1}{(t-s-h)^{2(\sigma+\gamma-\delta)}}ds\left(\int_H\abs{z}_{\delta}^4 dz\right),
\end{align*}
where $C:=2C_\sigma C_\kappa C_{\delta+\eps,\sigma+\gamma}C_{\delta+\eps,\kappa+\gamma}$. Since $\sigma+\gamma-\delta<1/2$ by the H\"older inequality with $p,q>1$ such that $p^{-1}+q^{-1}=1$ and $p<1/2(\sigma+\gamma-\delta)$ we have
{\small
\begin{align*}
I_3&\leq  C^2 h^{2\sigma+2\kappa}\left(\int^{t-h}_0\frac{1}{(t-s)^{2(\kappa+\gamma-\delta)q}}\right)^{2/q}\left(\int^{t-h}_0\frac{1}{(t-h-s)^{2(\sigma+\gamma-\delta)p}}ds\right)^{2/p}\left(\int_H\abs{z}_{\delta}^4 dz\right)\\
&\leq C^2 \frac{T^{2+4\delta-4\gamma-2\sigma-2\kappa}}{(1+2\delta-2\gamma-2\kappa)(1+2\delta-2\gamma-2\sigma)}h^{2\sigma+2\kappa}\left(\int_H\abs{z}_{\delta}^4 dz\right),
\end{align*}
}
hence there exists $K_{\kappa,\sigma,\gamma,\delta,T}$ such that
\begin{align}\label{StimaI3}
I_3\leq  K_{\kappa,\sigma,\gamma,\delta,T} h^{2\kappa+2\sigma}\left(\int_H\abs{z}_{\delta}^4 dz\right),
\end{align}
where $\sigma,\kappa>0$, $\sigma+\gamma-\delta<1/2$.

We estimate $\E[\abs{B_2}^2_\gamma\abs{C_2}^2_\gamma]$. By \cite[Proposition 2.4]{Pes-Zab2013} we have 
\begin{align*}
\E\left[\abs{B_2}^2_\gamma\abs{C_2}^2_\gamma\right]=J_1+J_2+J_3
\end{align*}
where
\begin{align*}
&J_1:=\int^{t}_{t-h}\int_H\abs{\left(e^{(t+h-s)A}-e^{(t-s)A}\right)z}^2_\gamma ds\nu(dz)\int^{t}_{t-h}\int_H\abs{e^{(t-s)A}z}^2_\gamma ds\nu(dz);\\
&J_2:=2\left[\int^{t}_{t-h}\int_H \scal{\left(e^{(t+h-s)A}-e^{(t-s)A}\right)z}{e^{(t-s)A}z}_\gamma ds\nu(dz)\right]^2;\\
&J_3:=\int^{t}_{t-h}\int_H\abs{\left(e^{(t+h-s)A}-e^{(t-s)A}\right)z}^2_\gamma \abs{e^{(t-s)A}z}^2_\gamma ds\nu(dz).
\end{align*}
We estimate $J_1$. In the same way of \eqref{stima2} and \eqref{stima5} we obtain
\begin{align}\label{StimaJ1}
J_1\leq K_{\theta,\gamma}K_{\gamma}h^{2+4\delta-4\gamma}\left(\int_{H}\abs{z}^2_\delta\nu(dz)\right)^2,
\end{align}
where $K_{\theta,\gamma,\delta}$ and $K_{\gamma,\delta}$ are the same constants in \eqref{stima2} and \eqref{stima5}, respectively.

 We estimate $J_2$. Let $\zeta>0$ be such that $2\gamma+\zeta<1$. By \eqref{analitic3} and \eqref{E1}, we have
\begin{align*}
J_2\leq C^2_\zeta C^2_{\delta,\gamma+\zeta} C_{\delta,\gamma}^2h^{2\zeta}\left(\int_{t-h}^t\frac{1}{(t-s)^{2\gamma+\zeta-2\delta}}ds\right)^2\left(\int_{H}\abs{z}^2_\delta\nu(dz)\right)^2,
\end{align*}
since $\zeta>0$ is arbitrary small, there exists $K''_{\gamma,\delta}>0$ such that
\begin{align}\label{StimaJ2}
J_2\leq K''_{\gamma,\delta}h^{2+4\delta-4\gamma}\left(\int_{H}\abs{z}^2_{\delta}\nu(dz)\right)^2. 
\end{align} 
Finally we estimate $J_3$. Let $\zeta>0$ such that $4\gamma+2\zeta-4\eps-4\delta<1$. By \eqref{analitic3} and \eqref{E1}, we have
\begin{align*}
J_3\leq C^2_\theta C^2_{\eps+\delta,\gamma+\zeta} C_{\eps+\delta,\zeta}^2h^{2\zeta}\int_{t-h}^t\frac{1}{(t-s)^{4\gamma+2\zeta-4\eps-4\delta}}ds\left(\int_{H}\abs{z}^4_{\eps+\delta}\nu(dz)\right),
\end{align*}
since $\zeta>0$ is arbitrary small, there exists $K_{\eps,\gamma,\delta,T}>0$ such that 
\begin{align}\label{StimaJ3}
J_3\leq K_{\eps,\gamma,T}h^{1+4\eps+4\delta-4\gamma}\left(\int_{H}\abs{z}^4_{\delta+\eps}\nu(dz)\right).
\end{align} 
By \eqref{eps-int}, \eqref{StimaFF}, \eqref{stima1}, \eqref{stima2}, \eqref{stima3}, \eqref{stima4}, \eqref{stima5}, \eqref{StimaI1}, \eqref{StimaI2}, \eqref{StimaI3}, \eqref{StimaJ1}, \eqref{StimaJ2} and \eqref{StimaJ3} there exists a constant $K>0$ independent of $h$ such that
\begin{align*}
\E[\abs{L_A(t+h)-L_A(t)}^2_\gamma\abs{L_A(t)-L_A(t-h)}^2_\gamma]&\leq K\bigg(h^{1+2\kappa-2\gamma+2\delta}+h^{1+2\sigma-2\gamma+2\delta}\\
&+h^{2\sigma+2\kappa}+h^{2-4\gamma+4\delta}+h^{1+4\eps+4\delta-4\gamma}\bigg),
\end{align*}
where $\delta>0$, $0<\eps\leq 1/4$, $0\leq \gamma<\eps+\delta$, $\sigma,\kappa>0$ and $\sigma+\gamma-\delta<1/2$. Since $\gamma-\delta<1/4$ then we can choose $\sigma>\gamma-\delta$ and $\kappa>\max\{1/2-\sigma,\gamma-\delta\}$ such that \eqref{stimaF} holds true.\end{proof}

\begin{rmk}\label{ErZabPes}
In this remark, we list the steps in \cite[Proof of Theorem 3.1]{Pes-Zab2013} which are not completely clear to us. We underline that the coefficient $\delta$ in this paper correspond to the coefficient $\rho$ in  \cite[Proof of Theorem 3.1]{Pes-Zab2013}.
\begin{enumerate}
\item On page $26$, in formula (27), the exponent $1-2\rho$ is not greater than $1$.
\item On page $28$, in the estimate for $I_2$, the exponent $2-4\rho$ is greater than $1$ if and only if $\rho<1/4$. However, in the statement of \cite[Theorem 3.1]{Pes-Zab2013}, it is only assumed that $\rho<1/2$.
\end{enumerate}
We believe that point (1) is just a slight inaccuracy. Indeed, in the proof presented in this paper, this issue does not appear. However, we were not able to fix point (2) In fact, in our statement, if $\delta=0$ then $\gamma$ is smaller than $1/4$.
\end{rmk}

The following proposition is an immediate consequence of classical Sobolev embedding and Proposition \ref{PZvariante}.

\begin{prop}\label{PropCarlo-Levy}
Let $d\in\N$, $H=L^2([0,1]^d)$ and $A$ be the realization of the Laplacian operator with Dirichlet boundary conditions. Assume that $\nu(H\backslash H_{1/4+\delta})=0$ and
\begin{equation}\label{condMo}
\int_H \abs{z}_\delta^2+\abs{z}^4_{1/4+\delta}\nu(dz)<+\infty.
\end{equation}
The following two assertions hold true.
\begin{enumerate}
\item $\{L_A(t)\}_{t\geq 0}$ is a càdlàg process in $H_\gamma$ with $\gamma<1/4+\delta$.
\item Let $\delta>\frac{d-1}{4}$. Then $\{L_A(t)\}_{t\geq 0}$ is a càdlàg process in $C([0,1]^d)$ and
\[
\E\left[\sup_{t\in [0,T]}\norm{L_A(t)}^q_{C([0,1]^d)}\right]<+\infty,\qquad T>0,\, q\in [1,4).
\]
\end{enumerate}

\end{prop}

\begin{rmk}\label{BvsL}
Propositions \ref{PropCarlo}(1) and \ref{PropCarlo-Levy}(1) (with $\delta=0$) they might seem similar. However in Proposition \ref{PropCarlo} $\{W(t)\}_{t\geq 0}$ is a $H$-cylindrical Wiener process, namely a Wiener process in $H_\eta$ for every $\eta<-1/4$ instead in Proposition \ref{PropCarlo-Levy} $\{L(t)\}_{t\geq 0}$ is a Lévy process in  $H$. On the other hand, both Propositions \ref{PropCarlo} and \ref{PropCarlo-Levy} show an improvement in the regularity of the trajectories of the stochastic convolution processes. 

\end{rmk}


The next example shows that condition \eqref{eps-int} does not imply that $\{L(t)\}_{t\geq 0}$ takes value in $H_{\delta+\eps}$. This fact highlights regularization properties of stochastic convolution, similar to those observed in the cylindrical Wiener  case.

\begin{ex}[Example 3.4 of \cite{Pes-Zab2013}]\label{ExPZ}
In this example we revisit in our framework \cite[Example 3.4]{Pes-Zab2013}. Let $H=l^2$ and let $\{e_n\}_{n\in\N}$ be an orthonormal basis of $H$. Let $A$ be a linear diagonal  operator defined by
\[
Ae_n=-n^2e_n,\qquad n\in\N, 
\]
hence, for every $\rho\geq 0$, we have
\[
H_\rho:={\rm Dom}((-A)^\rho)=\left\lbrace x=\sum_{n\in\N}x_ne_n\,:\, \sum_{n\in\N}x_n^2n^{2\rho}<+\infty\right\rbrace.
\]
 Let $\delta \ge 0$ and let $\delta_1$ be the Dirac measure concentrated in $1$.
We consider the Lévy measure
\[
\nu=\sum_{n\in\N}n^{-k}\delta_1
\]
for some $k \ge 0.$ 
In particular $\nu$ is the Lévy measure associated with a cylindrical Lévy process $\{L(t)\}_{t\geq 0}$ defined by
\[
L(t)=\sum_{n\in\N}n^{-k}l_n(t)e_n,
\]
where $\{l_n\}_{n\in\N}$ are one dimensional independent Poisson processes with intensity $1$. In this example \eqref{eps-int} reads as
\begin{equation}\label{eps-diag}
\sum_{n\in\N} n^{2(\delta-k)}+n^{4(\eps+\delta-k)}<+\infty,
\end{equation}
for some $0<\eps\leq 1/4$. If $k>1/2+\delta$ then \eqref{eps-diag} hold true with $\eps=1/4$ and so by Proposition \ref{PZvariante} the process $\{L_A(t)\}_{t\geq 0}$ is càdlàg in $H_\gamma$ for every $\gamma<1/4+\delta$. However by \cite[Proposition 2.6(i)]{Pes-Zab2013} $\{L(t)\}_{t\geq 0}$ is a Lévy process in $H_\gamma$ if and only if
\[
\sum_{n\in\N}\min\{n^{2(\gamma-k)},1\}<+\infty,
\]
namely $k>1/2+\gamma$. So, in this case, if we choose $\gamma\in (\delta,\delta+1/4)$ and $k\in (1/2+\delta,1/2+\gamma)$ then $\{L(t)\}_{t\geq 0}$ is not a Lévy process in $H_\gamma$ but  $\{L_A(t)\}_{t\geq 0}$ is càdlàg in $H_\gamma$.
\end{ex}

We conclude this section with a discussion about $\alpha$-stable process with $\alpha\in (0,2)$. Assume that there exists an orthonormal basis $\{e_n\}_{n\in\N}$ of $H$ consisting of eigenvectors of $A$ and let $\{\lambda_n\}_{n\in\N}$ be the corresponding eigenvalues. Let $\{\sigma_n\}_{n\in\N}$ be a sequence of positive number. We consider the Lévy process $\{L(t)\}_{t\geq 0}$ defined by
\begin{equation}\label{Alphastabile}
L(t):=\sum_{n\in\N}l_n(t)e_n,\qquad t>0,
\end{equation}
where $l_1,l_2,\ldots$ are real independent $\alpha$-stable process with $\alpha\in (0,2)$ having associated Levy measure $\nu_n$ given by
\[
\nu_n(d\xi)=C\frac{\sigma_n}{|\xi|^{\alpha+1}}d\xi,
\]
for some constant $C>0$, see also \cite{Pri-Zab2011}. It is not  possible to apply Proposition \ref{PZvariante} in the case where $\{L(t)\}_{t\geq 0}$ is defined by \eqref{Alphastabile}, and, even using a localisation method as in \cite[Corollary 3.2]{Pes-Zab2013}, it is not possible prove a regularity improvement for the stochastic convolution $\{L_A(t)\}_{t\geq 0}$, see \cite[Remark 3.5]{Pes-Zab2013} for further details. Indeed we have the following result in this case.

\begin{prop}[Proposition 2.3 of \cite{Liu-Zha2012}]\label{ProLiu}
Let $\{L(t)\}_{t\geq 0}$ be defined by \eqref{Alphastabile} with $\alpha\in (0,2)$ and let $\gamma\geq 0$. Then the following assertions are equivalent.
\begin{enumerate}
\item $\{L(t)\}_{t\geq 0}$ is a Lévy process on $H_\gamma$.
\vspace{2pt}
\item $\sum_{n\in\N} \abs{\sigma_n\lambda_n^\gamma}^{\alpha}<+\infty$.
\vspace{2pt}
\item $\{L_A(t)\}_{t\geq 0}$ is càdlàg in $H_\gamma$.
\end{enumerate}

\end{prop}

\section{The semilinear stochastic problem}\label{WP-SPDE}
In this section, we will study the well-posedness of the main equation addressed in this paper. 
Let $T>0$. We consider the following stochastic evolution equation
\begin{gather}\label{eqFO}
\eqsys{
dX(t,x)=\big(AX(t)+F(X(t))\big)dt+ RdW(t)+dL(t), & t\in [0,T]; \\
X(0)=x\in {H},
}
\end{gather}
where the coefficients are defined by the following assumptions. Below we will  provide  examples in which the assumptions are satisfied.

\begin{hyp}\label{Hyp} $ $\\
\begin{enumerate}[\rm(i)]
\item $H$ is a real separable Hilbert space and $E$ is a real separable Banach space continuously embedded in $H$.

\item $R:H\rightarrow H$ is a linear bounded operator.

\item\label{HHypA} $A:{\rm Dom}(A)\subseteq H\rightarrow H$ is a linear negative self-adjoint operator. The parts $A_E$ of $A$ in $E$ is a linear dissipative operator that generates an analytic semigroup $\{e^{tA}\}_{t\geq 0}$ on $E$ (see Subsections \ref{DD} and \ref{SemiSemi} for the definitions).

\item\label{HHypF} $F:E\rightarrow H$ is a measurable function and it maps bounded sets of $E$ in bounded sets of ${H}$. Let $F_E$ be the part of $F$ in $E$. The domain ${\rm Dom}(F_E)$ is a non-empty Borel subset of $E$. Moreover $F-\zeta_F\Id$ is $m$-dissipative in ${H}$ and $F_{E}-\zeta_F\Id$ is $m$-dissipative in $E$ (see Definition \ref{defi-dissi}).

\item\label{HHypZ} $\{W(t)\}_{t\geq 0}$ is a $H$-cylindrical Wiener process and $\{L(t)\}_{t\geq 0}$ is a pure-jump Lévy process on $H$ such that the process $\{Z_A(t)\}_{t\geq 0}$ defined by
\[
Z_A(t):=\int_0^te^{(t-s)A}RdW(s)+\int_0^t e^{(t-s)A}dL(t),
\]
is a càdlàg process in $E$.
Moreover, we assume that $\{ F(Z_A(s))\}_{s \in [0,T]}$ is a well-defined  $\{\mathcal{F}_t\}_{t\geq 0}$-adapted $E$-valued process with integrable trajectories, i.e., 
 \begin{equation}\label{IntZZ}
\int^T_0\norm{F(Z_A(s))}_Eds<+\infty,   ,\quad \mathbb{P}{\rm -a.s.}
\end{equation}
\end{enumerate} 
\end{hyp}

\begin{rmk}
We note that the dissipative assumptions on $A$ and $F$ are equivalent to the following two conditions:
\begin{enumerate}[\rm(a)]
\item there exists $\zeta_A\in\R$ such that $A-\zeta_A I$ is dissipative in $H$ and $A_E-\zeta_A I$ is dissipative in $E$,
\item  there exists $\zeta_F\in\R$ such that $F-\zeta_F I$ is dissipative in $H$ and $F_{E}-\zeta_F I$ is dissipative in $E$.
\end{enumerate}
In this case, we replace coefficients $A$ and $F$ with $\widetilde{A}=A-\zeta_A I$ and $\widetilde{F}=F+\zeta_A I$, respectively. In this way $\widetilde{A}$ and $\widetilde{F}$ verify Hypotheses \ref{Hyp} with $\zeta=\zeta_A+\zeta_F$.
\end{rmk}

\begin{ex}[$A$ satisfying Hypotheses \ref{Hyp}\eqref{HHypA}]\label{HypA}
Let ${H}=L^2(\mathcal{O},\lambda)$ where $\mathcal{O}$ is a open bounded subset of $\R^d$ having regular boundary and $\lambda$ is the Lebesgue measure. Let $A:{\rm Dom}(A)\subseteq{H}\rightarrow{H}$ be the realization of the Laplacian operator in ${H}$ with Dirichlet, Neumann or Robin Boundary conditions. Then $A$ verifies Hypotheses \ref{Hyp}\eqref{HHypA} with both $E=C(\overline{\mathcal{O}})$ and $E=L^{q}(\mathcal{O})$ for every $q\geq 2$.
\end{ex}

\begin{rmk}
Similarly to \cite[Hypotheses 7.3]{Dap-Zab2014}, we have not assumed that ${\rm Dom}(A_E)$ is dense in $E$. In the case where $A$ is a realization of the Laplacian operator, this fact allows us to consider the case $E = C(\overline{\mathcal{O}})$; indeed, not every choice of boundary conditions leads to a realization of the Laplacian with dense domain in $E$.
\end{rmk}

\begin{rmk}
We emphasize that the auxiliary result used to study \eqref{eqFO} holds under more general assumptions than those in  Hypothesis \ref{Hyp}, see Proposition \ref{proyos}. In particular, it could allow to consider operators \( A \) that are not necessarily self-adjoint. This would enable us to consider cases where \( A \) is the realization of a more general second-order differential operator; see, for instance, \cite[Section 6.1]{Cer2001}. 
\end{rmk}

\begin{ex}[$F$ satisfying Hypotheses \ref{Hyp}\eqref{HHypF}]\label{ExF}
In the same framework of Example \ref{HypA}, we consider a function $b:\overline{\mathcal{O}}\times \R\rightarrow\R$ defined by
\begin{equation}\label{bbb}
b(\xi,s):=-C_{2m+1}(\xi)s^{2m+1}+\sum^{2m}_{k=0}C_k(\xi)s^k, \qquad \xi\in \overline{\mathcal{O}},\ s\in\R,
\end{equation}
where $C_0,\ldots,C_{2m}$ are continuous functions from $\overline{\mathcal{O}}$ to $\R$ and $C_{2m+1}:\overline{\mathcal{O}}\rightarrow\R$ is a continuous and strictly positive function. We consider the Nemytskii operator $F:E\subseteq H\rightarrow H$ defined by 
\begin{equation}\label{Nemy}
F(x)(\xi):=b(\xi,x(\xi)),\qquad x\in C(\overline{\mathcal{O}}),\ \xi\in \overline{\mathcal{O}}.
\end{equation}
$F$ verifies Hypotheses \ref{Hyp}\eqref{HHypF} both with $E=C(\overline{\mathcal{O}})$ (see for instance \cite[Section 6.1.1]{Cer2001}) and with $E=L^{2(2m+1)}(\mathcal{O})$ (see for instance \cite[Chapter 4]{Dap2004}). To be precise in this case both $F:E\rightarrow H$ is continuous.
\end{ex}

\begin{ex}[$R$ satisfying Hypotheses \ref{Hyp}\eqref{HHypZ}]

In the same framework of Example \ref{HypA}. Let $F$ be defined by \eqref{Nemy}. If $E=C(\overline{\mathcal{O}})$ then $F_E(C(\overline{\mathcal{O}}))\subseteq C(\overline{\mathcal{O}})$ and $F_E:E\rightarrow E$ is locally Lipschitz. If $E=L^{2(2m+1)}(\mathcal{O})$ then $F_E(L^{2(2m+1)^2}(\mathcal{O}))\subseteq L^{2(2m+1)}(\mathcal{O})$ and $F_E:L^{2(2m+1)^2}\rightarrow E$ is continuous.\\
{ Choosing $R=(-A)^{-\gamma}$ with $\gamma\geq 0$ and $\{L(t)\}_{t\geq 0}$ Lévy process on $H_\delta={\rm Dom}((-A)^{\delta})$ with $\delta\geq 0$. In Propositions \ref{PropCarlo} and \ref{PropCarlo-Levy} are collected some sufficient conditions on $\gamma$ and $\delta$ to guarantee that $\{Z_A(t)\}_{t\geq 0}$ is a $C(\overline{\mathcal{O}})$-valued càdlàg process and so also $L^{q}(\mathcal{O})$-valued càdlàg process for every $q\geq 1$. See also Proposition \ref{ProLiu} for the $\alpha$-stable case.} \\
These  arguments guarantee that under the assumptions on $\delta$ given by Propositions \ref{PropCarlo}, \ref{PropCarlo-Levy} and \ref{ProLiu} our Hypotheses \ref{Hyp}\eqref{HHypZ} hold both with $E=C(\overline{\mathcal{O}})$ and with $E=L^{2(2m+1)}(\mathcal{O})$ (recall that, ${\mathbb P}-a.s.$, $Z_A$ is an $L^{2(2m+1)^2}(\mathcal{O})-$valued càdlàg process). 
\\
We underline that Hypotheses \ref{Hyp}\eqref{HHypZ} is similar to the one in \cite[(H2) page 183]{Pes-Zab2007}.

\end{ex}

In this section we will prove that for every $x\in E$ and $T>0$ the SPDE \eqref{eqFO} has unique mild solution in the following sense.

\begin{defi} 
Let $x\in E$ and $T>0$. We call mild solution to \eqref{eqFO} any $E$-valued $\{\mathcal{F}_t\}_{t\geq 0}$-adapted càdlàg process $\{X(t,x)\}_{t\in [0,T]}$ such that, for every $t\in [0,T]$, we have
\begin{align}\label{mildF}
 X(t,x)=e^{tA}x+\int_0^te^{(t-s)A}F(X(s,x))ds+Z_A(t),\quad\mathbb{P}{\rm-a.s.}
\end{align}
We say that the mild solution of \eqref{eqFO} is unique if  $\{X_1(t,x)\}_{t\in [0,T]}$, $\{X_2(t,x)\}_{t\in [0,T]}$ are two mild solutions to \eqref{eqFO} defined on the same complete filtered probability space then it satisfy
\[
X_1(t,x)=X_2(t,x),\quad \forall t\in [0,T], \quad \mathbb{P}{\rm-a.s.}
\]
\end{defi}

Hypotheses \ref{Hyp} are not sufficient to ensure the existence and uniqueness of a mild solution; an additional regularity assumption on the perturbation \( F \) is required. However, this assumption depends on the properties of the space \( E \); see Example \ref{ExF}. Below, we present two alternative assumptions that guarantee the well-posedness of \eqref{eqFO}.

\begin{hyp}\label{Hyp1C}
Assume that Hypotheses \ref{Hyp} holds true, $F(E)\subseteq E$ and $F_E=F_{|E}:E{\rightarrow} E$ is locally Lipschitz, namely it is Lipschitz continuous on the bounded sets of $E$.
\end{hyp}
\begin{ex}[$F$ satisfying Hypotheses \ref{Hyp1C}]\label{ExC}
In the same framework of Example \ref{ExF}. $F$ verifies Hypotheses \ref{Hyp1C} with $E=C(\overline{\mathcal{O}})$, see \cite[Section 6.1.1]{Cer2001}.
\end{ex}

\begin{hyp}\label{Hyp2L}
Assume that Hypotheses \ref{Hyp} holds true, $E$ is reflexive and there exists $x_0\in E$ such that $F(x_0)=\zeta_Fx_0$. Moreover the following implication holds true:
If $\{x_n\}_{n\in\N}\subseteq E$ is uniformly bounded in $E$ and converges strongly in ${H}$ to $x\in E$, then for every $h\in E$ we have
\begin{equation}\label{convergenza-deboleV}
|\scal{F(x_n)-F(x)}{h}_{H}|\rightarrow 0.
\end{equation}
\end{hyp}

\begin{ex}[$F$ satisfying Hypotheses \ref{Hyp2L}]\label{ExLp}
We consider the same framework as in Example \ref{ExF} with $E=L^{2(2m+1)}(\mathcal{O})$. Clearly, Hypotheses \ref{Hyp2L} are more technical than Hypotheses \ref{Hyp1C}. Unfortunately, $F$ cannot verify Hypotheses \ref{Hyp1C} with $E=L^{2(2m+1)}(\mathcal{O})$ but it verifies Hypotheses \ref{Hyp2L}. In the case where $b$ is defined by \eqref{bbb} with $C_0,\ldots,C_{2m+1}$ constants and $C_0=0$, $F$ is given by \eqref{Nemy}, $E=L^{2(2m+1)}(\mathcal{O})$,
we prove that \eqref{convergenza-deboleV} holds.\\
Let $\{x_n\}_{n\in\N}\subseteq E$ and $x\in E$ as in Hypotheses \ref{Hyp2L}. For every $h\in E$ and $n\in\N$ we have
\begin{align*}
|\scal{F(x_n)-F(x)}{h}_H| \leq K_1\sum^{2m+1}_{k=1}\int_\mathcal{O}|x_n(\xi)^k-x(\xi)^k||h(\xi)|d\xi,
\end{align*}
where $K_1=\max\left\{C_1,\ldots C_{2m+1}\right\}$. By factorizing the polynomials $(a^k-b^k)$ with $k\in\N$ and by the Young inequality, there exists a constant $K_2>0$ such that
\begin{align*}
|\scal{F(x_n)-F(x)}{h}_H| \leq K_2\sum^{2m+1}_{k=1}\int_\mathcal{O}|x_n(\xi)-x(\xi)||h(\xi)|(|x_n(\xi)|^{k-1}+|x(\xi)|^{k-1})d\xi,
\end{align*}
applying the H\"older inequality (with $p=q=2$) we obtain
\begin{align*}
|\scal{F(x_n)-F(x)}{h}_H| &\leq K_2\abs{x_n-x}_{L^2}\sum^{2m+1}_{k=1}\left(\int_\mathcal{O}|h(\xi)|^2(|x_n(\xi)|^{k-1}+|x(\xi)|^{k-1})^2d\xi\right)^{1/2}
\end{align*}
applying again the H\"older inequality (with $p=(2m+1)$ and $q=(2m+1)/2m$) we obtain
{\small
\begin{align*}
|\scal{F(x_n)-F(x)}{h}_H| &\leq K_2\abs{x_n-x}_{L^2}\norm{h}_{L^{2(2m+1)}}\sum^{2m+1}_{k=1}\left(\int_\mathcal{O}(|x_n(\xi)|^{k-1}+|x(\xi)|^{k-1})^{\frac{2m+1}{m}}d\xi\right)^{\frac{m}{2m+1}}.
\end{align*}
}
By the Young inequality, there exists a constant $K_3>0$ such that
{\small
\begin{align*}
|\scal{F(x_n)-F(x)}{h}_H| \leq K_3\abs{x_n-x}_{L^2}\norm{h}_{L^{2(2m+1)}}\sum^{2m+1}_{k=1}\Bigg[&\left(\int_\mathcal{O}|x_n(\xi)|^{\frac{(k-1)(2m+1)}{m}}d\xi\right)^{\frac{m}{2m+1}}\\
&+\left(\int_\mathcal{O}|x(\xi)|^{\frac{(k-1)(2m+1)}{m}}d\xi\right)^{\frac{m}{2m+1}}\Bigg],
\end{align*}
}
noting that $(k-1)/2m\leq 1$ for every $k\leq 2m+1$, by the Jensen inequality we obtain
{\small
\begin{align}
|\scal{F(x_n)-F(x)}{h}_H| &\leq K_3\abs{x_n-x}_{L^2} \norm{h}_{L^{2(2m+1)}} \sum^{2m+1}_{k=1}\left(\norm{x_n}_{L^{2(2m+1)}}^{k-1}+\norm{x}_{L^{2(2m+1)}}^{k-1}\right),\label{stima-convergenza}
\end{align}
}
and so that, since $\{x_n\}_{n\in\N}$ is equibounded in $E=L^{2(2m+1)}(\mathcal{O})$ and converges strongly in $H=L^{2}(\mathcal{O})$ to $x\in E$, then \eqref{stima-convergenza} yields \eqref{convergenza-deboleV}.

\end{ex}

\begin{rmk}
Hypotheses \ref{Hyp2L} represents the abstract framework corresponding to the specific equation studied in \cite{Mar-Roc2010}. In this case, in the proof of Proposition \ref{proyos}, we have added some details that were partially omitted in \cite{Mar-Roc2010}.
\end{rmk}

Now we can state the main result of this paper about the well-posedness of \eqref{eqFO} in the case $x\in E$. We note that this results is based on Proposition \ref{proyos} (see also Chapter 10 in \cite{Pes-Zab2007}).
\begin{thm}\label{solMild}
Assume that either Hypotheses \ref{Hyp1C} or Hypotheses \ref{Hyp2L} holds. For every $T>0$ and $x\in E$, \eqref{eqFO} has unique $E$-valued càdlàg mild solution $\{X(t,x)\}_{t\in [0,T]}$. Moreover for every $t>0$ and $x,z\in E$ we have
{\small
\begin{align}
&\abs{X(t,x)}_{H}\leq e^{\zeta t}\abs{x}_{H}+\int_0^te^{\zeta (t-s)}\left(\abs{F(Z_A(s))}_{H}+2\max\{0,\zeta\}\abs{Z_A(s)}_{H}\right)ds+\abs{Z_A(t)}_{H},\;\; \mathbb{P}{\rm-a.s.}\label{stindXX}\\
&\norm{X(t,x)}_E\leq e^{\zeta t}\norm{x}_E+\int_0^te^{\zeta (t-s)}\left(\norm{F(Z_A(s))}_E+2\max\{0,\zeta\}\norm{Z_A(s)}_E\right)ds+\norm{Z_A(t)}_E,\;\;\mathbb{P}{\rm-a.s.}\label{stindEX}	\\
&\abs{X(t,x)-X(t,z)}_{H}\leq  e^{\zeta t}\abs{x-z}_{H},\;\;\mathbb{P}{\rm-a.s.}\label{lipXX}\\
&\norm{X(t,x)-X(t,z)}_E\leq  e^{\zeta t}\norm{x-z}_E,\;\;\mathbb{P}{\rm-a.s.}\label{lipEX},
\end{align}
where $\zeta$ is the constant defined in \eqref{zeta}.
}
\end{thm}
\begin{proof}
Fix $x\in E$ and $T>0$. Fix an $H$-cylindrical Wiener process $\{W(t)\}_{t\geq 0}$ and a Lévy process $\{L(t)\}_{t\geq 0}$ defined on a same complete filtered probability space $(\Omega,\mathcal{F},\{\mathcal{F}_t\}_{t\geq 0},\mathbb{P})$. By Hypotheses \ref{Hyp}  there exists $\Omega_0\in \mathcal{F} $ such that $\mathbb{P}(\Omega_0)=1$ and for every $\omega\in\Omega_0$ the function $t\rightarrow Z_A(t)(\omega)$ is càdlàg from $[0,T]$ to $E$.\\

\textbf{Uniqueness} If $\{X_1(t,x)\}_{t\in [0,T]}$ and $\{X_2(t,x)\}_{t\in [0,T]}$ are two mild solutions to \eqref{eqFO} defined on $(\Omega,\mathcal{F},\{\mathcal{F}_t\}_{t\geq 0},\mathbb{P})$ then for every $\omega\in\Omega_0$ both functions $t\in [0,T]\rightarrow X_1(\omega)-Z_A(\omega)$ and $t\in [0,T]\rightarrow X_2(\omega)-Z_A(\omega)$ are mild solutions to
\begin{gather}\label{eqV}
\eqsys{
\dfrac{dy}{dt}(t)=Ay(t)+F(y(t)+f(t^-)), & t>0;\\
y(0)=x\in E
}
\end{gather}
with $f=Z_A(\omega)$, so by the uniqueness established in Proposition \ref{proyos} we have
\[
\sup_{t\in [0,T]}\abs{X_1(t,x)-X_2(t,x)}_H=0,\qquad \mathbb{P}-{\rm a.s.}
\]

\textbf{Existence} Let $\{Y(t,x)\}_{t\in [0,T]}$ be the process defined for every $\omega\in\Omega_0$  by 
\begin{equation*}
Y(\cdot,x)(\omega):=y(\cdot,x),
\end{equation*}
where $y(\cdot,x)$ is the unique mild solution to \eqref{eqV} with $f(\cdot):=Z_A(\cdot)(\omega)$ given by Proposition \ref{proyos}. Obviously the process $\{X(t,x)\}_{t\in [0,T]}$ defined by
\begin{equation}\label{deX}
X(t,x):=Y(t,x)+Z_A(t),
\end{equation}
solves the mild form of \eqref{eqFO} $\mathbb{P}$-a.s. { Proposition \ref{proyos} guarantee that $y(\cdot,x)\in C((0,T];E)$, hence, since we have assumed that $\{Z_A(t)\}_{t\geq 0}$ is càdlàg then $\{X(t,x)\}_{t\in [0,T]}$ is càdlàg in $E$. The pathwise estimate \eqref{stindXX}, \eqref{stindEX}, \eqref{lipXX} and \eqref{lipEX} follows directly from the estimates on $y(\cdot,x)$ given by Proposition \ref{proyos} and the definition \eqref{deX} of $X(t,x)$.}\\
It remains to prove that $\{X(t,x)\}_{t\in [0,T]}$ is adapted to $\{\mathcal{F}_t\}_{t\geq 0}$. We note that if we prove that $\{Y(t,x)\}_{t\in [0,T]}$ is adapted to $\{\mathcal{F}_t\}_{t\geq 0}$ then also $\{X(t,x)\}_{t\in [0,T]}$ is adapted to $\{\mathcal{F}_t\}_{t\geq 0}$ and the proof is concluded. For every small $\delta,\theta\geq 0$ we study the following SPDE
\begin{gather*}
\eqsys{
\dfrac{dY_{\delta,\theta}}{dt}(t,x)=A_\theta Y_{\delta,\theta}(t,x)+F_{\delta}(Y_{\delta,\theta}(t,x)+Z_A(t^-)), & t>0;\\
Y_{\delta,\theta}(0,x)=x\in E,
}
\end{gather*}
where $F_\delta$ and $A_\theta$ are the Yosida approximants of $F$ and $A$ defined in Section \ref{Sec:Yosida}. By Lemma \ref{soldt} and the same arguments used above, there exists a unique $E$-valued càdlàg process $\{Y_{\delta,\theta}(t,x)\}_{\in  [0,T]}$ such that
\[
Y_{\delta,\theta}(t,x)=x+\int_0^t\left[A_\theta +\zeta_A\Id \right]Y_{\delta,\theta}(s,x)ds+ \int_0^tF_\delta(Y_{\delta,\theta}(s,x)+Z_A(s))ds,\qquad \mathbb{P}-{\rm a.s.}
\]
Moreover by \eqref{convdtE} we have
\[
\lim_{\delta\rightarrow 0}\lim_{\theta\rightarrow 0}\sup_{t\in [0,T]}\abs{Y_{\delta,\theta}(t,x)-Y(t,x)}_H=0,\qquad \mathbb{P}-{\rm a.s.},
\]
so if $\{Y_{\delta,\theta}(t,x)\}_{t\in [0,T]}$ is adapted to $\{\mathcal{F}_t\}_{t\geq 0}$ then also $\{Y(t,x)\}_{t\geq 0}$ is adapted to $\{\mathcal{F}_t\}_{t\geq 0}$. For every $m\in\N$ we consider the processes $\{Y^{(m)}_{\delta,\theta}(t,x)\}_{t\in [0,T]}$ defined, for every $t\in [0,T]$, by
{\small
\begin{align*}
&Y^{(0)}_{\delta,\theta}(t,x)=x,\\
&Y^{(m+1)}_{\delta,\theta}(t,x)=x+\int_0^t A_\theta Y^{(m)}_{\delta,\theta}(s,x)ds+ \int_0^tF_\delta(Y^{(m)}_{\delta,\theta}(s,x)+Z_A(s))ds,\qquad m>1.
\end{align*}
}
By construction $\{Y^{(m)}_{\delta,\theta}(t,x)\}_{t\in [0,T]}$ is adapted to $\{\mathcal{F}_t\}_{t\geq 0}$ for every $m\in\N$. Moreover since $F_\delta:H\rightarrow H$ is Lipschitz continuous and $A_\theta\in\mathcal{L}(H)$, using the classical Picard iteration scheme it is not difficult to show that
\[
\lim_{m\rightarrow +\infty}\sup_{t\in [0,T]}\abs{Y^{m}_{\delta,\theta}(t,x)-Y_{\delta,\theta}(t,x)}_H=0,\qquad \mathbb{P}-{\rm a.s.},
\]
and so $\{Y_{\delta,\theta}(t,x)\}_{t\in [0,T]}$ is adapted to $\{\mathcal{F}_t\}_{t\geq 0}$.
\end{proof}

\begin{ex}\label{ExNuovo}
We observe that under Hypotheses \ref{Hyp1C}, the function \( F \) is smooth from \( E \) to \( E \), so it is not necessary to consider the Hilbert space \( H \) in order to apply Theorem \ref{solMild}. This fact allows us to consider functions $F$ that are well-defined only on $E$. We show an interesting example. We consider the same framework of \eqref{ExF} with $E=C([0,1])$, but as $F:E\rightarrow E$ we consider
\[
F(x)(\xi):=b(\xi,x(\xi))+g(\max_{s\in [0,\xi]}|x(s)|),\qquad x\in C(\overline{\mathcal{O}}),\ \xi\in [0,1]
\] 
where $b$ is given by \eqref{bbb} and $g$ is a Lipschitz continuous function, see \cite{Cer-Dap-Fla2013} for the properties of $F$ when $b\equiv 0$. In this case $F$ is well defined only on $E=C([0,1])$ and Hypotheses \ref{Hyp1C} hold true, so Theorem \ref{solMild} is applicable.
\end{ex}

\begin{rmk}\label{remark-finale}$ $
\begin{enumerate}
\item Since Theorem \ref{solMild} does not assume the existence of moments of any order, it can also be applied to $\alpha$-stable processes with $\alpha\in (0,2)$. Indeed in the framework of  Example \ref{ExC} or \ref{ExLp}, choosing $\mathcal{O}=[0,1]$, $R=\Id$, and $\{L(t)\}_{t\geq 0}$ a Lévy process in $H_\delta$ with $\delta>1/4$ defined by \eqref{Alphastabile}, then by Propositions \ref{PropCarlo} and \ref{ProLiu} all the assumptions of Theorem \ref{solMild} are verified.

\item Taking into account Example \ref{ExPZ}, we underline that Theorem \ref{solMild} ensure that the mild solution to \eqref{eqFO} is càdlàg in $E$ even in cases where the $\{L(t)\}_{t\geq 0}$ is not càdlàg in $E$.

\item Assuming additional conditions on the Lévy process $\{L(t)\}_{t\geq 0}$, by estimates \eqref{stindXX} and \eqref{stindEX}, it is possible to deduce estimate on the moments of the mild solution to \eqref{eqFO}, similar to the one in \cite{Mar-Pre-Roc2010,Mar-Roc2010}.
\end{enumerate}
\end{rmk}

Now we exploit the density of $E$ in ${H}$ to define a process $\{X(t,x)\}_{t\in [0,T]}$ for every $x\in H$.

\begin{thm}\label{limmild}
Assume that either Hypotheses \ref{Hyp1C} or Hypotheses \ref{Hyp2L} holds. For every $x\in H$ and $T>0$ there exists a unique $H$-valued càdlàg process $\{X(t,x)\}_{t\in [0,T]}$ such that for every $\{x_n\}_{n\in\N}\subseteq E$ converging to $x$ we have
\begin{align}\label{CX}
&\lim_{n{\rightarrow}+\infty}\sup_{t\in [0,T]}\abs{X(\cdot,x_n)-X(\cdot,x)}_{H}=0,\qquad \mathbb{P}-{\rm a.s.},
\end{align}
where $\{X(t,x_n)\}_{t\in [0,T]}$ is the unique mild solution of \eqref{eqFO} with initial datum $x_n$. In addition for every $t\in [0,T]$ and $x,z\in {H}$ we have
{\small
\begin{align}
&\abs{X(t,x)}_{H}\leq e^{\zeta t}\abs{x}_{H}+\int_0^te^{\zeta (t-s)}\left(\abs{F(Z_A(s))}_{H}+2\max\{0,\zeta\}\abs{Z_A(s)}_{H}\right)ds+|Z_A(t)|_H,\;\;\mathbb{P}{\rm-a.s.}\label{SX}\\
&\abs{X(t,x)-X(t,z)}_{H}\leq  e^{\zeta t}\abs{x-z}_{H},\;\;\mathbb{P}{\rm-a.s.}\label{LX},
\end{align}
}
where $\zeta$ is the constant defined in \eqref{zeta}.
\end{thm}

\begin{proof}
\noindent Since $E$ is dense in $H$, for every $x\in{H}$ there exists a sequence $\{x_n\}_{n\in \N}\subseteq E$ such that
\[
\lim_{n{\rightarrow}+\infty} \abs{x_n-x}_{H}=0.
\]
By \eqref{lipXX}, for every $T>0$ and $n_1,n_2\in\N$, we have
\[
\lim_{n_1,n_2{\rightarrow}+\infty}\sup_{t\in [0,T]}\abs{X(\cdot,x_{n_1})-X(\cdot,x_{n_2})}_{H}=0,\quad \mathbb{P}{\rm-a.s.},
\]
we denote by $\{X(t,x)\}_{t\in [0,T]}$ the limit of $\{X(t,x_n)\}_{t\in [0,T]}$. We underline that by \eqref{lipXX} the limit $\{X(t,x)\}_{t\in [0,T]}$ does not depend on the sequence $\{x_n\}_{n\in \N}$, so it is the unique $H$-valued càdlàg process that verifies \eqref{CX}. Finally \eqref{stindXX} and \eqref{lipXX} yield \eqref{SX} and \eqref{LX}.
\end{proof}

\begin{defi}
For every $x\in{H}$ we call generalized mild solution of \eqref{eqFO} the limit \\$\{X(t,x)\}_{t\in [0,T]}$ given by Proposition \ref{limmild}. 
\end{defi}

\begin{rmk}\label{Open-question}
In this final remark, we list some open questions that we believe might be interesting.
\begin{enumerate}
\item Proposition \ref{valoriHgamma} and Remark \ref{BvsL} lead us to believe that the assumptions in Proposition \ref{PZvariante} are not optimal. However, utilizing tests such as the Gikhman-Skorokhod test \cite[Section III.4]{Gik-Sko1990} or the Bezandry--Fernique test \cite{Bez-Fer1990}, we were unable to obtain better results than those in Proposition \ref{PZvariante}.

\item It would be interesting to understand if it is possible to weaken the monotonicity assumptions on $F$ with assumptions similar to those used in \cite{Gyo-Kri1980}. In this direction, we point out the work \cite{Brz-Liu-Zhu2014} where, however, they consider variational solutions instead of mild ones.

\item This point is not an open question but a possible future project. We believe that by using the estimates in Theorem \ref{solMild}, it is possible to find asymptotic results analogous to those of the Brownian case contained in \cite[Chapter 8]{Cer2001}. We underline that in the framework of Example \ref{ExLp} and Remark \ref{RmkMPR}, some asymptotic results are proven in \cite{Mar-Pre-Roc2010,Mar-Roc2010}.
\end{enumerate}
\end{rmk}

\appendix

\section{Dissipative mappings and semigroups}
 The first part of the appendix is devoted to recalling some preliminary results concerning subdifferentials, dissipative mappings, and semigroup theory, which will be instrumental for the second part, where we prove an existence and uniqueness result.
\subsection{Left differentiable and dissipative mappings}\label{DD}
We refer to \cite[Appendix A]{Cer2001}, \cite[Appendix D]{Dap-Zab2014} and \cite[Section 10.1]{Pes-Zab2007} for all the results in this section.
Let $\mathcal{K}$ be a separable Banach space. For every $x\in\mathcal{K}$, we define the subdifferential $\partial \norm{x}_{\mathcal{K}}$ of $\norm{\cdot}_{\mathcal{K}}$ at $x\in{\mathcal{K}}$ as
\[
\partial \norm{x}:=\{ x^*\in {\mathcal{K}}^*\; |\; \norm{x+y}_{\mathcal{K}}\geq\norm{x}_{\mathcal{K}}+\dscal{y}{x^*}{\mathcal{K}},\;\forall\; y\in{\mathcal{K}}\}.
\]
$\partial \norm{x}_{\mathcal{K}}$ is non empty closed and convex set and for every $x\neq 0$ we have
\[
\partial \norm{x}=\{ x^*\in {\mathcal{K}}^*\; |\; \dscal{x}{x^*}{\mathcal{K}}=\norm{x}_{\mathcal{K}},\; \norm{x^*}_{\mathcal{K}^*}=1\}.
\]

Let $t_0,t_1 \in\R$ and let $f:[t_0,t_1]{\rightarrow}{\mathcal{K}}$ be a continuous function. We say that $f$ is left-differentiable if for every $t\in (t_0,t_1]$ there exists $L\in\R$ such that
\begin{equation*}
\dfrac{d^-f(t)}{dt}:=\lim_{\epsilon{\rightarrow} 0^+}\dfrac{f(t)-f(t-\epsilon)}{\epsilon}=L.
\end{equation*}
Let $u:[t_0,t_1]{\rightarrow}{\mathcal{K}}$ be a continuous and left-differentiable function. By \cite[Proposition 10.1]{Pes-Zab2007} the function $\gamma:=\norm{u}_{\mathcal{K}}:(t_0,t_1]{\rightarrow} [0,+\infty)$ is left-differentiable at any $t\in (t_0,t_1]$ and
\begin{equation}\label{Ldiff}
\dfrac{d^-\gamma}{dt}(t):=\lim_{h{\rightarrow}0^-}\dfrac{\gamma(t+h)-\gamma(t)}{h}=\min\{\dscal{u'(t)}{x^*}{\mathcal{K}}\;:\; x^*\in\partial\norm{u(t)}_{\mathcal{K}} \}.
\end{equation}
Moreover, let $b\in\R$ and let $g:[t_0,t_1]{\rightarrow} [0,+\infty)$ be a continuous function. If
\[
\dfrac{d^-\gamma}{dt}(t)\leq b\gamma(t)+g(t), 
\]
then, by \cite[Proposition A.4]{Big-Fer2024}, for every $t\in [t_0,t_1]$, we have
\begin{equation}\label{varofcost}
\gamma(t)\leq e^{b(t-t_0)}\gamma(t_0)+\int^t_{t_0}e^{b(t-s)}g(s)ds,\quad t\in [t_0,t_1].
\end{equation} 

Using the notion of subdifferential we have the following useful charaterization for the dissipative maps (see Definition \ref{defi-dissi}).

\begin{prop}\label{dissiBanach}
Let $f:{\rm Dom}(f)\subseteq \mathcal{K}{\rightarrow}\mathcal{K}$. $f$ is dissipative if and only if, for every $x,y\in {\rm Dom}(f)$ there exists $z^*\in\partial\norm{x-y}$ such that
\begin{equation}\label{disban}
\dscal{f(x)-f(y)}{z^*}{\mathcal{K}}\leq 0.
\end{equation}
If $\mathcal{K}$ is a Hilbert space \eqref{disban} reads as
\begin{equation*}
\scal{f(x)-f(y)}{x-y}_{\mathcal{K}}\leq 0.
\end{equation*}
\end{prop}

\subsection{The Yosida approximating}\label{Sec:Yosida}
In this section we introduce a useful regularizing sequence for dissipative functions (see \cite[Appendix A]{Cer2001}, \cite[Appendix D]{Dap-Zab2014} and \cite[Section 10.1]{Pes-Zab2007}).

Let ${\mathcal{K}}$ be a separable Banach space and let $F:{\rm Dom}(F)\subseteq{\mathcal{K}}{\rightarrow}{\mathcal{K}}$ be a possibly non linear function. We assume that there exists $\zeta_F\in\R$ such that $F-\zeta_F\Id$ is m-dissipative. By the m-dissipativity, for every $\delta>0$ the map $z\rightarrow  z-\delta(F(y)-\zeta_F y)$ is bijective, so for every  $x\in{\mathcal{K}}$, there exists a unique $J_\delta(x)\in {\rm Dom}(F)$ such that
\begin{align}\label{eq_YO}
J_\delta(x)-\delta (F(J_\delta(x))-\zeta_F J_\delta(x))=x.
\end{align}
We define $F_\delta:{\mathcal{K}}{\rightarrow}{\mathcal{K}}$ as
\begin{align*}
F_\delta(x):=F(J_\delta(x)),\qquad x\in{\mathcal{K}},\ \delta>0.
\end{align*}

The following lemma is only a slight modification of the propositions in \cite[Appendix D]{Dap-Zab2014}, so we omit the proof since the computations are standard.

\begin{lemm}\label{Lemma_YO}
The following statements hold true.
\begin{align*}
\lim_{\delta{\rightarrow} 0}\norm{J_{\delta}(x)-x}_{\mathcal{K}}=0,\quad x\in {\rm Dom}(F).
\end{align*}
For every $0<\delta<|\zeta_F|^{-1}$ ($0<\delta<\infty$ if $\zeta_F=0$), the function $F_\delta-\zeta_F\Id_{{\mathcal{K}}}$ is dissipative on ${\mathcal{K}}$. Moreover for every $0<\delta<|\zeta_F|^{-1}$ it holds
\begin{align}
\norm{J_\delta(x)-x}_{\mathcal{K}}&\leq \delta\left(\|F(x)\|_K+\max\{0,\zeta_F\}\| x\|_{\mathcal{K}}\right),\qquad x\in {\rm Dom}(F);\label{cji}\\
\|F_{\delta}(x)\|_{\mathcal{K}}&\leq (1+\delta \max\{0,\zeta_F\})(\|F(x)\|_{\mathcal{K}}+2 \max\{0,\zeta_F\}\norm{x}_{\mathcal{K}}),\qquad x\in {\rm Dom}(F);\label{vy1}\\
\norm{F_\delta(x)-F_\delta(y)}_K&\leq \left(2\delta^{-1}+ \max\{0,\zeta_F\}\right)\norm{x-y}_K,\qquad x,y\in{\mathcal{K}}.\label{lipdeltaX}\\
\norm{J_\delta(x)-J_\delta(y)}_{\mathcal{K}} &\leq \norm{x-y}_K,\qquad x,y\in{\mathcal{K}}\label{lipJ}.
\end{align}
For every $x,y\in \mathcal{K}$ and $0<\delta,\tau<|\zeta_F|^{-1}$ we have
\begin{equation}\label{superyE}
\norm{J_\delta(x)-J_\tau(y)}_{\mathcal{K}}\leq \zeta_F\norm{x-y}_{\mathcal{K}}+ (\delta+\tau)(\norm{F(y)}_{\mathcal{K}}+\norm{F(x)}_{\mathcal{K}}+|\zeta_F|\norm{y}_{\mathcal{K}}+|\zeta_F|\norm{x}_{\mathcal{K}}).
\end{equation} 
If ${\mathcal{K}}$ is a Hilbert space, then for every $x,y\in {\rm Dom}(F)$ there exists $B:=B(\zeta_F)$ such that for every $\delta,\tau>0$ we have
\begin{equation}\label{supery}
\scal{F_{\delta}(x)-F_\tau(y)}{x-y}_{\mathcal{K}}\leq \zeta_F\norm{x-y}_{\mathcal{K}}^2+(\delta+\tau)B(\norm{F(x)}_{\mathcal{K}}^2+\norm{F(y)}_{\mathcal{K}}^2+\norm{x}_{\mathcal{K}}^2+\norm{y}_{\mathcal{K}}^2)
\end{equation}

\end{lemm}

\begin{rmk}\label{punto-fisso}
If we assume that there exists $x_0\in{\rm Dom}(F)$ such that 
\[
F(x_0)=\zeta_Fx_0,
\]
then by \eqref{eq_YO}, for every $0<\delta<|\zeta_F|^{-1}$ we have $J_\delta(x_0)=x_0$. Hence by \eqref{lipJ} for every $x\in{\mathcal{K}}$ we get
\begin{equation*}
\norm{J_\delta(x)}_{\mathcal{K}}\leq \norm{J_\delta(x)-J_\delta(x_0)}+\norm{x_0}_{\mathcal{K}}\leq \norm{x-x_0}_K+\norm{x_0}_{\mathcal{K}}\leq \norm{x}_{\mathcal{K}}+2\norm{x_0}_{\mathcal{K}}
\end{equation*}
\end{rmk}
\subsection{Semigroups}\label{SemiSemi} 
In this subsection we recall some basic definitions and results about semigroups theory. We refer to \cite[Chapter II]{Eng-Nag2006} and \cite[Chapter 2]{Lun1995}.

Let ${\mathcal{K}}$ be a separable Banach space. Let $B:{\rm Dom}(B)\subseteq {\mathcal{K}}{\rightarrow}{\mathcal{K}}$. We denote by $\rho(B)$ the resolvent set of $B$ and for $\lambda\in\rho(B)$ we denote by $R(\lambda,B):=(\lambda\Id_{\mathcal{K}}-B)^{-1}$ the resolvent operator of $B$. 

We consider the complexification of ${\mathcal{K}}$, and we still denote it by ${\mathcal{K}}$. Let $B:{\rm Dom}(B)\subseteq{\mathcal{K}}{\rightarrow}{\mathcal{K}}$ be a sectorial operator, namely there exist $M>0$, $\eta_0\in\R$ and $\theta_0\in (\pi/2,\pi]$ such that
\begin{equation*}
S_{0}:=\{\lambda\in\mathbb{C}\; |\; \lambda\neq \eta_0,\; \vert \mbox{arg}(\lambda-\eta_0)\vert<\theta_0 \}\subseteq\rho(B);
\end{equation*}
\begin{equation*}
\norm{R(\lambda,B)}_{\mathcal{L}({\mathcal{K}})}\leq\frac{M}{|\lambda-\eta_0|},\quad \forall\lambda\in S_0.
\end{equation*}  
We denote by $\{e^{tB}\}_{t\geq 0}$ the analytic semigroup defined by $B$ via Dunford integral. We underline that, in general, the definition of analytic semigroup does not require that the domain of $B$ is dense, see for instance \cite[Chapter 2]{Lun1995}. We recall some basic properties.
\begin{enumerate}
\item There exists $M_0>0$ such that for every $t>0$
\begin{equation}\label{anlitic1}
\norm{e^{tB}}_{\mathcal{L}({\mathcal{K}})}\leq M_0e^{\eta_0t}.
\end{equation}
\item Let $f(t)=e^{tB}$, we have 
\begin{equation}\label{anlitic4}
f\in C^{\infty}((0,+\infty),\mathcal{L}({\mathcal{K}})).
\end{equation} 
\end{enumerate}

If $B:{\rm Dom}(B)\subseteq {\mathcal{K}}{\rightarrow}{\mathcal{K}}$ is dissipative then the approximants $\{B_\delta\}_{\delta>0}\subseteq \mathcal{L}(E)$ defined in Proposition \ref{Lemma_YO} are the standard Yosida approximants of $B$. Hence by \cite[Proposition 2.1.2]{Lun1995} and the definition of $\{e^{tB}\}_{t\geq 0}$ via Dunford integral we have
\begin{equation*}
\lim_{\delta{\rightarrow} 0}\sup_{t\in [\epsilon,T]}\norm{e^{tB_\delta}-e^{tB}}_{\mathcal{L}({\mathcal{K}})}=0,\quad 0<\epsilon<T.
\end{equation*}
Since, for every $\delta>0$, $B_\delta$ is bounded and dissipative by the Lumer-Phillips theorem we have
\begin{equation*}
\norm{e^{tB_\delta}}_{\mathcal{L}({\mathcal{K}})}\leq 1,\quad t\geq 0.
\end{equation*}

\section{The deterministic equation with a dissipative drift term}\label{WP-PDE}
In the second part of the appendix, we will study the auxiliary evolution equation associated with \eqref{eqFO}, that we have used in Section \ref{WP-SPDE} to prove the main result of the paper. We aim for a result similar to the one in \cite[Section 7.2.3]{Dap-Zab2014} for the Brownian case. However, in the Lévy case, we encounter less regularity in time. Additionally, we provide some finer estimates.

In all this section we fix $T>0$ and $f:[0,T]{\rightarrow} E$ which is a  càdlàg function.
We consider the following evolution equation 
\begin{gather}\label{eq}
\eqsys{
\dfrac{dy}{dt}(t)=Ay(t)+F(y(t)+f(t^-)), & t>0;\\
y(0)=x\in E
}
\end{gather}
where $f(s^-):=\lim_{h{\rightarrow} s^-}f(h)$
and the coefficients satisfy the following assumptions.
\begin{hyp}\label{EU2}$ $
 
\begin{enumerate}[\rm(i)]
\item Let $H$ be a real separable Hilbert space and $E$ is a real separable Banach space continuously embedded in $H$.

\item\label{EU2.6} $F:E\subseteq H\rightarrow H$ is a measurable function  and it maps bounded sets of $E$ in bounded sets of ${H}$.

\item\label{EU2.5} $A$ is the infinitesimal generator of a strongly continuous and analytic semigroup $\{e^{tA}\}_{t\geq 0}$ on $H$. The part $A_E$ of $A$ in $E$ generates an analytic semigroup.

\item\label{EU2.4} There exist $\zeta_{A}, \zeta_F\in\R$ such that 
\begin{enumerate}
\item $A-\zeta_{A}\Id$ is dissipative in ${H}$ and $A_E-\zeta_{A}\Id$ is dissipative in $E$;
\item $F-\zeta_F\Id$ is $m$-dissipative in ${H}$ and $F_{E}-\zeta_F\Id$ is $m$-dissipative in $E$.
\end{enumerate}
We set 
\begin{equation}\label{zeta}
\zeta:=\zeta_{A}+\zeta_F.
\end{equation}

\item
The mapping $F \circ f: [0, T] \to E$ is well-defined and integrable, i.e.,
\[
\int^T_0\norm{F(f(s))}_Eds<+\infty.
\]
\end{enumerate}
\end{hyp}

By leveraging Hypotheses \ref{EU2}, we approximate problem \eqref{eq} using Yosida approximants of $A$ and $F$. Nevertheless, these assumptions alone are insufficient to establish that the limit of the solutions to the approximated problems converges to the a mild solution to \eqref{eq}. To achieve this, we must assume at least one of the following additional hypotheses.

\begin{hyp}\label{Hyp2}
Assume that Hypotheses \ref{EU2} holds true, $F(E)\subseteq E$ and $F_E=F_{|E}:E{\rightarrow} E$ is locally Lipschitz, namely it is Lipschitz continuous on the bounded sets of $E$.
\end{hyp}

\begin{hyp}\label{Hyp1}
Assume that Hypotheses \ref{EU2} holds true, $E$ is reflexive and there exists $x_0\in E$ such that $F(x_0)=\zeta_Fx_0$. Moreover the following implication holds true:
If $\{x_n\}_{n\in\N}\subseteq E$ is uniformly bounded in $E$ and converges strongly in ${H}$ to $x\in E$, then for every $h\in E$ we have
\begin{equation}\label{convergenza-debole}
|\scal{F(x_n)-F(x)}{h}_{H}|\rightarrow 0.
\end{equation}
\end{hyp}

In this section, under Hypotheses \ref{Hyp2} or \ref{Hyp1}, we will prove that for every $x\in E$ the PDE \eqref{eq} has unique mild solution $y(\cdot,x)\in C([0,T];{H})\cap C((0,T];E)$ in the following sense.
\begin{defi}
For every  $x\in E$ we call mild solution of \eqref{eq} a function $y(\cdot,x)$ such that, for every $t\in [0,T]$, we have

\begin{align}\label{Mild-det}
 y(t,x)=e^{tA}x+\int_0^te^{(t-s)A}F(y(s,x)+f(s))ds.
\end{align}
\end{defi}

We will prove the well-posedness of \eqref{eq} in Subsection \ref{Buona-positura}, however, to do this we need two technical lemmas that we will state and prove in the next Subsection \ref{Lemmi-tecnici}.

\subsection{Technical lemmas}\label{Lemmi-tecnici}

We still denote by $A$ and $F$ the part of $A$ and $F$ in $E$, respectively. Let $\{F_\delta \}_{\delta\geq 0}$ and $(A-\zeta_A\Id)_{\theta\geq 0}$ be the The Yosida approximants of $F$ and $(A-\zeta_A\Id)$, respectively. By Hypotheses \ref{EU2}, for every small $\theta,\delta\geq 0$, the functions $F_\delta$ and $(A-\zeta_A\Id)_\theta$ verify the statements of  Proposition \ref{Lemma_YO} both with ${\mathcal{K}}={H}$ and ${\mathcal{K}}=E$.

For every small $\theta,\delta\geq 0$ we introduce the approximate problem
\begin{gather}\label{eqdt}
\eqsys{
\dfrac{dy_{\delta,\theta}}{dt}(t)=\left[(A-\zeta_A\Id)_\theta+\zeta_A\Id\right]y_{\delta,\theta}(t)+F_{\delta}(y_{\delta,\theta}(t)+f(t^-)), & t\in [0,T];\\
y_{\delta,\theta}(0)=x\in E.
}
\end{gather}
If $f$ is a continuous function then it is possible to avoid the next proposition using \cite[Proposition 4.1.8]{Lun1995}.

\begin{prop}\label{soldt}
Assume that Hypotheses \ref{EU2} hold true. For every small $\theta,\delta\geq 0$ and\\ $x\in E$, the PDE \eqref{eqdt} has a unique mild solution $y_{\delta,\theta}(\cdot,x)\in C([0,T],E)$ and it is \\left-differentiable. Moreover for every small enough $\tau,\delta, \theta>0$, for every $x,z\in E$ and \\$t\in [0,T]$ we have
\begin{align}
&|y_{\delta,\theta}(t,x)|_{H}\leq e^{\zeta t}|x|_{H}+(1+\delta\max\{0,\zeta_F\})\int_0^te^{\zeta (t-s)}\left(\abs{F(f(s))}_{H}+2\max\{0,\zeta_F\}\abs{f(s)}_{H}\right)ds,\label{stidXth}\\
&\norm{y_{\delta,\theta}(t,x)}_E\leq e^{\zeta t}\norm{x}_E+(1+\delta\max\{0,\zeta_F\})\int_0^te^{\zeta (t-s)}\left(\norm{F(f(s))}_E+2\max\{0,\zeta_F\}\norm{f(s)}_E\right)ds,\label{stidEth}	\\
&\abs{y_{\delta,\theta}(t,x)-y_{\delta,\theta}(t,z)}_{H}\leq  e^{\zeta t}\abs{x-z}_{H}\label{lipdXth}\\
&\norm{y_{\delta,\theta}(t,x)-y_{\delta,\theta}(t,z)}_E\leq  e^{\zeta t}\norm{x-z}_E\label{lipdEth},\\
&\abs{y_{\delta,\theta}(t,x)-y_{\tau,\theta}(t,x)}_{H}^2\leq C(\delta+\tau)\label{lipdelta}
\end{align}
where $\zeta$ is the constant defined in \eqref{zeta} and $C:=C(A,\zeta,x,F,T)$ is a positive constant.
\end{prop}
\begin{proof}
We fix small $\theta,\delta>0$ and $x\in E$. In this proof we set by 
\[
A'_\theta:=\left[(A-\zeta_A\Id)_\theta+\zeta_A\Id\right].
\]
We stress that $A'_\theta\in\mathcal{L}(E)$. We consider the operator $V:C([0,T],E){\rightarrow} C([0,T],E)$ defined by
\[
V(y)(t):=e^{tA'_\theta}x+\int^t_0e^{(t-s)A'_\theta}F_{\delta}(y(s)+f(s))ds,\quad y\in C([0,T],E),\; t\in [0,T].
\]
Since $f$ is càdlàg then $s\rightarrow f(s^-)$ is càglàg so the function $s\rightarrow F_\delta(y(s)+f(s^-))$ is integrable in $[0,t]$.
$A'_\theta\in\mathcal{L}(E)$, so by \eqref{lipdeltaX} we have $V\left(C([0,T],E)\right)\subseteq C([0,T],E)$, and for every $y,z\in C([0,T],E)$ we have
\[
\norm{V(y)-V(z)}_{C([0,T],E)}\leq \left(\frac{2}{\delta}+\max\{0,\zeta_F\}\right)\left(\int^T_0\norm{e^{(t-s)A'_\theta}}_{\mathcal{L}(E)}ds\right)\norm{y-z}_{C([0,T],E)}.
\]
Hence for $T_0>0$ small enough, by the contraction mapping theorem the problem  \eqref{eqdt} has a unique mild solution $y_{\delta,\theta,T_0}(\cdot,x)\in C([0,T_0],E)$. We have proven that \eqref{eqdt} has unique mild solution in every interval $[0,T_0]$, $[T_0,2T_0]$,... , so, attaching the solutions in these intervals, we obtain a mild solution $y_{\delta,\theta,T}(\cdot,x)$ of \eqref{eqdt} in $[0,T]$. The uniqueness of $y_{\delta,\theta}(\cdot,x)$ follow immediately by the Lipschitzianity of $F_{\delta}$. 

Now we prove that $y_{\delta,\theta}(\cdot,x)$ is left-differentiable, namely
\begin{equation}\label{S0.1}
\lim_{h\rightarrow 0^+}\norm{\dfrac{y_{\delta,\theta}(t,x)-y_{\delta,\theta}(t-h,x)}{h}-A'_\theta y_{\delta,\theta}-F_{\delta}(y_{\delta,\theta}(t,x)+f(t^-))}_E=0.
\end{equation}
For every $t\in [0,T]$, $0<h<t$ and $x\in E$ we have
{\small
\begin{align}
\dfrac{y_{\delta,\theta}(t,x)-y_{\delta,\theta}(t-h,x)}{h}&=\frac{1}{h}\left(e^{tA'_\theta}x-e^{(t-h) A'_\theta}x\right)+\frac{1}{h}\int^t_0e^{(t-s)A'_\theta}F_{\delta}(y_{\delta,\theta}(s,x)+f(s))ds\notag\\
&-\frac{1}{h}\int^{t-h}_0e^{(t-h-s)A'_\theta}F_{\delta}(y_{\delta,\theta}(s,x)+f(s))ds\notag\\
&=\frac{1}{h}\left(e^{tA'_\theta}x-e^{(t-h) A'_\theta}x\right)+\frac{1}{h}\int^{t}_{t-h}e^{(t-s)A'_\theta}F_{\delta}(y_{\delta,\theta}(s,x)+f(s))ds\notag\\
&+\frac{1}{h}\int^{t-h}_0\left(e^{tA'_\theta}x-e^{(t-h) A'_\theta}x\right)e^{-sA'_\theta}F_{\delta}(y_{\delta,\theta}(s,x)+f(s))ds.\notag
\end{align}
}
Since $A'_\theta\in\mathcal{L}(E)$ then by standard arguments
\begin{align}
&\lim_{h{\rightarrow} 0^+}\norm{\left(\dfrac{e^{tA'_\theta}-e^{(t-h) A'_\theta}}{h}\right)\left(x+\int^{t-h}_0e^{-sA'_\theta}F_{\delta}(y_{\delta,\theta}(s,x)+f(s^-))ds\right)-A'_\theta y_{\delta,\theta}}_E=0.\label{S1}
\end{align}
Now we prove that 
\begin{equation}\label{S2}
\lim_{h{\rightarrow} 0^+}\norm{\frac{1}{h}\int^{t}_{t-h}e^{(t-s)A'_\theta}F_{\delta}(y_{\delta,\theta}(s,x)+f(s))ds-F_{\delta}(y_{\delta,\theta}(t,x)+f(t^-))}_E=0.
\end{equation}
Since $F_\delta$ is Lipschitz and $s\rightarrow f(s^-)$ is càglàg, then for every $\epsilon>0$ there exists $\delta>0$ such that for every $s\in (t-\delta,t)$ we have
\[
\norm{F_{\delta}(y_{\delta,\theta}(s,x)+f(s^-))ds-F_{\delta}(y_{\delta,\theta}(t,x)+f(t^-))}_E\leq \epsilon,
\] 
so for every $0<h<\delta$
\begin{align*}
&\norm{\frac{1}{h}\int^{t}_{t-h}e^{(t-s)A'_\theta}F_{\delta}(y_{\delta,\theta}(s,x)+f(s))ds-F_{\delta}(y_{\delta,\theta}(t,x)+f(t^-))}_E\\
&\leq \frac{1}{h}\int^{t}_{t-h}\norm{e^{(t-s)A'_\theta}F_{\delta}(y_{\delta,\theta}(s,x)+f(s))-F_{\delta}(y_{\delta,\theta}(t,x)+f(t^-))}_Eds \leq \epsilon,
\end{align*}
and \eqref{S2} holds true. Combining \eqref{S1} and \eqref{S2} we obtain \ref{S0.1} so for every $t\in [0,T]$ and $x\in E$ we have
\[
\dfrac{d^-y_{\delta,\theta}}{dt}(t,x)=A'_\theta y_{\delta,\theta}(s,x)+F_{\delta}(y_{\delta,\theta}(t,x)+f(t^-)),
\]

We prove \eqref{stidEth}. By \eqref{Ldiff}, Proposition \ref{Lemma_YO} and Hypotheses \ref{EU2}\eqref{EU2.4} there exists $y^*\in\partial \norm{y_{\delta,\theta}(t,x)}_E$ such that
\begin{align*}
\dfrac{d^-\norm{y_{\delta,\theta}(t,x)}_E}{dt}&\leq\dscal{\left[(A-\zeta_A\Id)_\theta+\zeta_A\Id\right] y_{\delta,\theta}(t,x)}{y^*}{E}+\dscal{F_{\delta}(y_{\delta,\theta}(t,x)+f(t^-))}{y^*}{E}\\
&=\dscal{\left[(A-\zeta_A\Id)_\theta+\zeta_A\Id\right] y_{\delta,\theta}(t,x)}{y^*}{E}\\
&+\dscal{F_{\delta}(y_{\delta,\theta}(t,x)+f(t^-))-F_{\delta}(f(t^-))}{y^*}{E}+\dscal{F_{\delta}(f(t^-))}{y^*}{E}\\
&\leq \zeta\norm{y_{\delta,\theta}(t,x)}_E+\norm{F_{\delta}(f(t^-))}_E,
\end{align*}
where $\zeta$ is the constant defined in \eqref{zeta}. By \eqref{varofcost} we get
\[
\norm{y_{\delta,\theta}(t,x)}_E\leq e^{\zeta t}\norm{x}_E+\int_0^te^{\zeta (t-s)}\norm{F_\delta(f(s))}_E ds,
\] 
finally \eqref{vy1} yields \eqref{stidXth}.

We prove \eqref{lipdEth}. By \eqref{Ldiff}, Proposition \ref{Lemma_YO} and Hypotheses \ref{EU2}\eqref{EU2.4} there exists $y^*\in\partial \norm{y_{\delta,\theta}(t,x)}_E$ such that
\begin{align*}
\dfrac{d^-\norm{y_{\delta,\theta}(t,x)-y_{\delta,\theta}(t,z)}_E}{dt}&\leq\dscal{\left[(A-\zeta_A\Id)_\theta+\zeta_A\Id\right] (y_{\delta,\theta}(t,x)-y_{\delta,\theta}(t,z))}{y^*}{E}\\
&+\dscal{F_{\delta}(y_{\delta,\theta}(t,x)+f(t^-))-F_{\delta}(y_{\delta,\theta}(t,z)+f(t^-))}{y^*}{E}\\
&\leq \zeta\norm{y_{\delta,\theta}(t,x)-y_{\delta,\theta}(t,z)}_E
\end{align*}
where $\zeta$ is the constant defined in \eqref{zeta}, so by \eqref{varofcost} we obtain \eqref{lipdEth}. 

Recalling that $(A-\zeta_A\Id)_\theta$ and $F_\delta-\zeta_F\Id$ are also dissipative in ${H}$, using the inner product of ${H}$ instead of the duality product of $E$, by the same procedure used above we obtain \eqref{stidXth} and \eqref{lipdXth}.

We prove \eqref{lipdelta}. By \eqref{Ldiff}, Proposition \ref{Lemma_YO} and \ref{EU2}\eqref{EU2.4}, for every $t\in [0,T]$ we have
{\small
\begin{align*}
\frac{1}{2}\dfrac{d\abs{y_{\delta,\theta}(t,x)-y_{\tau,\theta}(t,x)}_{H}^2}{dt}&\leq\scal{\left[(A-\zeta_A\Id)_\theta+\zeta_A\Id\right] (y_{\delta,\theta}(t,x)-y_{\tau,\theta}(t,x))}{y_{\delta,\theta}(t,x)-y_{\tau,\theta}(t,x)}_{H}\\
&+\scal{F_{\delta}(y_{\delta,\theta}(t,x)+f(t^-))-F_{\tau}(y_{\tau,\theta}(t,x)+f(t^-))}{y_{\delta,\theta}(t,x)-y_{\tau,\theta}(t,x)}_{H}\\
&\leq \zeta\abs{y_{\delta,\theta}(t,x)-y_{\tau,\theta}(t,x)}_{H}^2\\
&+(\delta+\tau)B\left(\abs{F(y_{\delta,\theta}(t,x)+f(t^-))}_{H}^2+\abs{F(y_{\tau,\theta}(t,x)+f(t^-))}_{H}^2\right)\\
&+(\delta+\tau)B\left(\abs{y_{\delta,\theta}(t,x)+f(t^-)}_{H}^2+\abs{y_{\tau,\theta}(t,x)+f(t^-)}_{H}^2\right)
\end{align*}
}
where $B$ is the constant in \eqref{supery}. By \eqref{varofcost} we obtain
{\small
\begin{align*} 
\abs{y_{\delta,\theta}(t,x)-y_{\tau,\theta}(t,x)}_{H}^2&\leq (\delta+\tau)B\int^t_0e^{2(t-s)\zeta}\left(\abs{F(y_{\delta,\theta}(s,x)+ f(s))}_{H}^2+\abs{F(y_{\tau,\theta}(s,x)+f(s))}_{H}^2\right)ds\\
&+(\delta+\tau)B\int^t_0e^{2(t-s)\zeta}\left(\abs{y_{\delta,\theta}(s,x)+f(s)}_{H}^2+\abs{y_{\tau,\theta}(s,x)+f(s)}_{H}^2\right)ds,
\end{align*}
}
hence by \eqref{stidEth} and Hypotheses \ref{EU2}(\ref{EU2.6}) we get \eqref{lipdelta}.\end{proof}

For every small $\delta\geq 0$ we study the mild solution of the following PDE
\begin{gather}\label{eqd}
\eqsys{
\dfrac{dy_{\delta}}{dt}(t,x)=Ay_{\delta}(t,x)+F_{\delta}(y_{\delta}(t,x)+f(t^-)), & t>0;\\
y_{\delta}(0,x)=x\in E.
}
\end{gather}

\begin{prop}
Assume that Hypotheses \ref{EU2} hold true. For every small $\delta>0$ and $x\in E$ there exists a function $y_\delta(\cdot,x)\in C([0,T],{H})\cap C_b((0,T],E)$ such that
\begin{equation}\label{convdtX}
\lim_{\theta{\rightarrow} 0}\norm{y_{\delta,\theta}(\cdot,x)-y_{\delta}(\cdot,x)}_{C([0,T],{H})}=0,
\end{equation}
\begin{equation}\label{convdtE}
\lim_{\theta{\rightarrow} 0}\norm{y_{\delta,\theta}(\cdot,x)-y_{\delta}(\cdot,x)}_{C([\epsilon,T],E)}=0,\quad \forall\;\epsilon>0,
\end{equation}
and $y_\delta(\cdot,x)$  is the unique mild solution of \eqref{eqd} in $C([0,T],{H})\cap C((0,T],E)$.

In addition for every small $\delta>0$, for every $x,z\in E$ and $t>0$ we have
{\small
\begin{align}
&\abs{y_{\delta}(t,x)}_{H}\leq e^{\zeta t}\abs{x}_{H}+(1+\delta\max\{0,\zeta_F\})\int_0^te^{\zeta (t-s)}\left(\abs{F(f(s))}_{H}+2\max\{0,\zeta_F\}\abs{f(s)}_{H}\right)ds,\label{stidX}\\
&\norm{y_{\delta}(t,x)}_E\leq e^{\zeta t}\norm{x}_E+(1+\delta\max\{0,\zeta_F\})\int_0^te^{\zeta (t-s)}\left(\norm{F(f(s))}_E+2\max\{0,\zeta_F\}\norm{f(s)}_E\right)ds,\label{stidE}	\\
&\abs{y_\delta(t,x)-y_\delta(t,z)}_{H}\leq  e^{\zeta t}\abs{x-z}_{H}\label{lipdX}\\
&\norm{y_\delta(t,x)-y_\delta(t,z)}_E\leq  e^{\zeta t}\norm{x-z}_E\label{lipdE},\\
&\abs{y_{\delta}(t,x)-y_{\tau}(t,x)}_{H}^2\leq C(\delta+\tau).\label{stimaconv}
\end{align}
}
where $\zeta$ is the constant defined in \eqref{zeta} and $C$ is the constant given by \eqref{lipdelta}
\end{prop}
\begin{proof}
Since $F_\delta$ is Lipschitz continuous and $\{e^{tA}\}_{t\geq 0}$ is a strongly continuous semigroup in ${H}$, by the same arguments used in the proof of Proposition \ref{soldt} for every $x\in E$ there exists a unique mild solution $y_\delta(\cdot,x)$ to \eqref{eqd} in $C([0,T],{H})$, namely
\[
y_\delta(t,x)=e^{tA}x+\int^t_0e^{(t-s)A}F_\delta(y_\delta(s,x)+f(s))ds,\quad t>0.
\]
Since $\{e^{tA_E}\}$ is an analytic semigroup in $E$ (but not necessary with dense domain), by \eqref{anlitic4}, $y(\cdot,x)\in C((0,T],E)$. 

\noindent We prove \eqref{convdtE}. Let $0<\epsilon<T$. For every small $\theta>0$ and for every $t\in (0,T]$ we have
\begin{align*}
y_{\delta,\theta}(t,x)-y_{\delta}(t,x)&\leq \left(e^{t(A-\zeta_A\Id)_\theta+\zeta_A\Id}-e^{tA}\right)x\phantom{aaaaaaaaaaa}\\
&+\int^t_0e^{(t-s)\zeta_A}e^{(t-s)(A-\zeta_A\Id)_\theta}F_\delta(y_{\delta,\theta}(s,x)+f(s))ds\\
&-\int^t_0 e^{(t-s)A}F_\delta(y_{\delta}(s,x)+f(s))ds\\
&=e^{\zeta_At}\left(e^{t(A-\zeta_A\Id)_\theta}-e^{t(A-\zeta_A\Id)}\right)x\\
&+\int^t_0e^{(t-s)\zeta_A}\left(e^{(t-s)(A-\zeta_A\Id)_\theta}-e^{(t-s)(A-\zeta_A\Id)}\right)F_\delta(y_{\delta,\theta}(s,x)+f(s))ds\\
&+\int^t_0e^{(t-s)A}\left(F_\delta(y_{\delta,\theta}(s,x)+f(s))-F_\delta(y_{\delta}(s,x)+f(s))\right)ds\\
\end{align*}
hence, recalling that $F_\delta$ is Lipschitz continuous, we have
{\small
\begin{align*}
\norm{y_{\delta,\theta}(t,x)-y_{\delta}(t,x)}_E & \leq  e^{\zeta_At}\norm{e^{t(A-\zeta_A\Id)_\theta}-e^{t(A-\zeta_A\Id)}}_{\mathcal{L}(E)}\norm{x}_E\\
&+\int^t_0 e^{(t-s)\zeta_A} \norm{e^{(t-s)(A-\zeta_A\Id)_\theta}-e^{(t-s)(A-\zeta_A\Id)}}_{\mathcal{L}(E)}\norm{F_\delta(y_{\delta,\theta}(s,x)+f(s))}_E ds\\
&+\int^t_0\norm{e^{(t-s)A}}_{\mathcal{L}(E)}\norm{y_{\delta,\theta}(s,x)-y_{\delta}(s,x))}_Eds.
\end{align*}
}
By \eqref{anlitic1}, \eqref{stidEth} and the Gronwall inequality there exists a constant $C:=C(A,F_\delta,x,\zeta,T)$ such that
\[
\sup_{t\in [\epsilon,T]}\norm{y_{\delta,\theta}(t,x)-y_{\delta}(t,x)}_E\leq C\sup_{t\in [\epsilon,T]}\norm{e^{t(A-\zeta_A\Id)_\theta}-e^{t(A-\zeta_A\Id)}}_{\mathcal{L}(E)},
\]
so \eqref{convdtE} holds true. Since $y_\delta\in C([0,T],H)$ and $E$ is continuously embedded in $H$, then also \eqref{convdtX} holds true. 

Finally letting $\theta{\rightarrow} 0$ in \eqref{stidXth}, \eqref{stidEth}, \eqref{lipdXth}, \eqref{lipdEth} and \eqref{lipdelta} we obtain \eqref{stidX}, \eqref{stidE}, \eqref{lipdX}, \eqref{lipdE} and \eqref{stimaconv}, respectively.\end{proof}

\subsection{Well-posedness of \eqref{eq}}\label{Buona-positura}

Now we can prove that \eqref{eq} has unique mild solution.

\begin{prop}\label{proyos}
Assume that Hypotheses \ref{Hyp2} or Hypotheses \ref{Hyp1}  hold true. For every $x\in E$, there exists a function $y(\cdot,x)\in C([0,T],{H})\cap C_b((0,T],E)$ such that
\begin{equation}\label{convdX}
\lim_{\delta{\rightarrow} 0}\norm{y_{\delta}(\cdot,x)-y(\cdot,x)}_{C([0,T],{H})}=0,\quad \forall\; T>0,\;
\end{equation}
and $y(\cdot,x)$ is the unique mild solution of \eqref{eq}. Moreover for every $x,z\in E$ and $t>0$ we have
{\small
\begin{align}
&\abs{y(t,x)}_{H}\leq e^{\zeta t}\abs{x}_{H}+\int_0^te^{\zeta (t-s)}\left(\abs{F(f(s))}_{H}+2\max\{0,\zeta_F\}\abs{f(s)}_{H}\right)ds,\label{stiX}\\
&\norm{y(t,x)}_E\leq e^{\zeta t}\norm{x}_E+\int_0^te^{\zeta (t-s)}\left(\norm{F(f(s))}_E+2\max\{0,\zeta_F\}\norm{f(s)}_E\right)ds,\label{stiE}	\\
&\abs{y(t,x)-y(t,z)}_{H}\leq  e^{\zeta t}\abs{x-z}_{H}\label{lipX}\\
&\norm{y(t,x)-y(t,z)}_E\leq  e^{\zeta t}\norm{x-z}_E\label{lipE},
\end{align}
}
where $\zeta$ is the constant defined in \eqref{zeta}. In addition if Hypotheses \ref{Hyp2} holds true then for every $x\in E$
\begin{equation}\label{convE}
\lim_{\theta{\rightarrow} 0}\norm{y_{\delta}(\cdot,x)-y(\cdot,x)}_{C([\epsilon,T],E)}=0,\quad \forall\;\epsilon>0,
\end{equation} 
\end{prop}
\begin{proof}
By \eqref{stimaconv}, there exists $y(\cdot,x)\in C([0,T],{H})$ such that \eqref{convdtX} holds true. Moreover letting $\delta\rightarrow 0$ in \eqref{stidX} and \eqref{lipdX} we get \eqref{stiX} and \eqref{lipX}.
To prove that $y(\cdot,x)\in C([0,T],{H})$ solves the mild form of \eqref{eq} we have to split the proof in two cases.\\
 
\textbf{Case 1: Hypotheses \ref{Hyp1} hold true.} 
 
Let $T>0$, $x\in E$. By the reflexivity of $E$ and \eqref{stidE}, for every $t\in [0,T]$, there exists a subsequence of $\{y_\delta(t,x)\}_{\delta>0}\subseteq E$ weakly convergent in $E$, so by \eqref{stidE}, \eqref{lipdE} and the lower semicontinuity of the norm \eqref{stiE} and \eqref{lipE} are verified. By \eqref{vy1} with ${\mathcal{K}}={H}$, \eqref{stidE} and Hypotheses \ref{EU2}\eqref{EU2.6}, for every $s\in [0,T]$ and $x\in E$ the sequence $\{F_\delta(y_\delta(s,x)+f(s^-))\}_{\delta\geq 0}$ is uniformly bounded in ${H}$, so it has a subsequence (still denoted by $\{F_\delta(y_\delta(s,x)+f(s^-))\}_{\delta\geq 0}$) weakly convergent to $L\in{H}$, namely for every $h\in{H}$
\begin{equation}\label{convL}
\lim_{\delta\rightarrow 0}\abs{\scal{F_\delta(y_\delta(s,x)+f(s^-))-L}{h}}=0.
\end{equation}
We prove that \eqref{convL} holds true with $L=F(y(s,x)+f(s^-))$.
Fix $s\in [0,T]$ and $x\in E$. For every small $\delta>0$ we have
\begin{align*}
\abs{J_{\delta}(y_\delta(s,x)+f(s^-))-y(s,x)-f(s^-)}_{H}&\leq \abs{J_{\delta}(y_\delta(s,x)+f(s^-))-y_\delta(s,x)-f(s^-)}_{H}\\
&+\abs{y_\delta(s,x)-y(s,x)}_{H}.
\end{align*}
By \eqref{cji} and \eqref{stidE}, we get
\begin{align*}
\abs{J_{\delta}(y_\delta(s,x)+f(s^-))-y(s,x)-f(s^-)}_{H}&\leq \delta(\abs{F(y_\delta(s,x)+f(s^-))}_{H}+\abs{\zeta}_F\abs{y_\delta(s,x)+f(s^-)}_{H})\\
&+\abs{y_\delta(s,x)-y(s,x)}_{H},
\end{align*}
By Hypotheses \ref{EU2}\eqref{EU2.6}, \eqref{stidE}  and \eqref{convdX} we obtain
\begin{align}\label{convJ}
\lim_{\delta{\rightarrow} 0}\abs{J_{\delta}(y_\delta(s,x)+f(s^-))-y(s,x)-f(s^-)}_{H}=0
\end{align}
By Hypotheses \ref{Hyp1} and Remark \ref{punto-fisso}, for every $\delta>0$ we have
\begin{equation}\label{unifE}
\norm{J_{\delta}(y_\delta(s,x)+f(s^-))}_E\leq \norm{y_\delta(s,x)}_E+\norm{f(s^-)}_E+2\norm{x_0}_{E},
\end{equation}
where $x_0\in E$ is given by Hypotheses \ref{Hyp1}. By \eqref{stidE} and \eqref{unifE}, there exists a constant $M>0$ independent on $\delta$ such that
\begin{equation}\label{unifEV}
\norm{J_{\delta}(y_\delta(s,x)+f(s^-))}_E\leq M,
\end{equation} 
so the sequence $\{J_{\delta}(y_\delta(s,x)+f(s^-))\}_{\delta>0}$ is uniformly bounded in $E$ and by \eqref{convJ} converges strongly $H$. Hence by Hypotheses \ref{Hyp1} (see \eqref{convergenza-debole}) for every $h\in E$ we have
\[
\abs{\scal{F(J_{\delta}(y_\delta(s,x)+f(s^-)))-F(y(s,x)+f(s^-))}{h}_{H}}\rightarrow 0,
\]
and since $E$ is dense in $H$ we can conclude that $L=F(y(s,x)+f(s^-))$ in \eqref{convL}, namely
\begin{equation*}
\lim_{\delta\rightarrow 0}\abs{\scal{F_\delta(y_\delta(s,x)+f(s^-))-F(y(s,x)+f(s^-)}{h}}=0,\qquad \forall\,h\in{H}.
\end{equation*}
Moreover by Hypotheses \ref{EU2}\eqref{EU2.6}, \eqref{vy1} (with $\mathcal{K}=H$) and \eqref{unifEV}, there exists $K$ such that for every small $\delta>0$ we have
\begin{equation}\label{unifEH}
\norm{F(J_{\delta}(y_\delta(s,x)+f(s^-)))}_H\leq K.
\end{equation} 
For every $t\in [0,T]$ and $h\in{H}$, we have 
\[
\scal{y_\delta(t,x)}{h}_{H}=\scal{e^{tA}x}{h}_{H}+\int_0^te^{(t-s)A}F_\delta(y_\delta(s,x)+f(s))ds,
\]
letting $\delta\rightarrow 0$ by \eqref{convJ}, \eqref{unifEH} and the dominated convergence theorem we obtain
\[
\scal{y(t,x)}{h}_{H}=\scal{e^{tA}x}{h}_{H}+\int_0^t\scal{e^{(t-s)A}F(y(s,x)+f(s))}{h}_{H} ds,
\]
 so by the arbitrariness of $h\in{H}$ the function $y(\cdot,x)$ solves the mild form of \eqref{eq}.

\textbf{Case 2: Hypotheses \ref{Hyp2} hold true.}  

We begin to prove \eqref{convE}. Let $T>0$, $x\in E$. For every small $\delta,\tau>0$ and $t\in [0,T]$
we have
\begin{align*}
\norm{y_{\delta}(t,x)-y_{\tau}(t,x)}_E&=\norm{\int^t_0e^{(t-s)A}\left(F_\delta(y_{\delta}(s,x)+f(s))-F_\tau(y_{\tau}(s,x)+f(s))\right)ds}_E,
\end{align*}
by \eqref{anlitic1} there exist $\eta_0\in\R$ and $M_0>0$ such that
\begin{align*}
\norm{y_{\delta}(t,x)-y_{\tau}(t,x)}_E\leq M_0\int^t_0e^{\eta_0(t-s)}\norm{F(J_\delta(y_{\delta}(s,x)+f(s)))-F(J_\tau(y_{\tau}(s,x)+f(s)))}_Eds
\end{align*}

By Hypotheses \ref{Hyp2} $F(E)\subseteq E$ and so ${\rm Dom}(F_E)=E$. Hence by \eqref{cji}  and \eqref{stidE} there exists a constant $R:=R(\zeta,T,x,F)>0$ such that, for every small $\delta,\theta>0$ and $t\in [0,T]$ we have
\begin{equation*}
\norm{J_\delta(y_{\delta}(t,x)+f(t^-))}_E+\norm{J_\tau(y_{\tau}(t,x)+f(t^-))}_E\leq R.
\end{equation*}
Let $L_R>0$ be the Lipschitz constant of $F$ on $B_E(0,R)$ (there exists by Hypotheses \ref{Hyp2}). By \eqref{superyE} (with ${\rm Dom}(F_E)=E$) we obtain
{\small
\begin{align*}
\norm{y_{\delta}(t,x)-y_{\tau}(t,x)}_E&
\leq M_0L_R\int^t_0e^{\eta_0(t-s)}\norm{J_\delta(y_{\delta}(s,x)+f(s))-J_\tau(y_{\tau}(s,x)+f(s))}_Eds\\
&\leq M_0L_R\zeta\int^t_0e^{\eta_0(t-s)}\norm{y_{\delta}(s,x)-y_{\tau}(s,x)}_Eds\\
&+M_0L_R(\delta-\tau)\int^t_0e^{\eta_0(t-s)}\norm{F(y_{\delta,\theta}(s,x)+f(s))}_E+\norm{F(y_{\tau,\theta}(s,x)+f(s))}_Eds\\
&+M_0L_R(\delta-\tau)\int^t_0e^{\eta_0(t-s)}\norm{y_{\delta,\theta}(s,x)+f(s)}_E+\norm{y_{\tau,\theta}(s,x)+f(s)}_Eds.
\end{align*}
}
By Hypotheses \ref{Hyp2}, \eqref{stidE} and the Gronwall inequality there exists a constant $C:=C(\zeta,x,F,T,A)$ such that
\begin{align*}
\norm{y_{\delta}(t,x)-y_{\tau}(t,x)}_E\leq C(\delta+\tau),
\end{align*}
so there exists $y'(\cdot,x)\in C((0,T],E)$ that verifies \eqref{convE}. Since $E$ is continuously embedded in ${H}$ we have that $y'(\cdot,x)=y(\cdot,x)$. Moreover by \eqref{stidE}, \eqref{lipdE} and\eqref{convE} we obtain \eqref{stiE} and \eqref{lipE}.

To prove that $y(\cdot,x)$ is a mild solution of \eqref{eq} it is sufficient to prove that for every $t\in [0,T]$ we have
\begin{equation}\label{mildFFE}
\lim_{\delta{\rightarrow} 0}\int^t_0\norm{F_\delta(y_\delta(s,x)+f(s))-F(y(s,x)+f(s))}_Eds=0
\end{equation}
Recalling that ${\rm Dom}(F)=E$, by \eqref{cji}, \eqref{lipJ} and \eqref{convE} we have
\begin{align*}
\lim_{\delta\rightarrow 0}\norm{J_{\delta}(y_\delta(t,x)+f(t^-))- y(t,x)-f(t^-)}_E\leq 
\end{align*}
since $F_{|E}:E{\rightarrow} E$ is continuous, then
\begin{align}\label{convFE}
\lim_{\delta{\rightarrow} 0}\norm{F(J_{\delta}(y_\delta(t,x)+f(t^-)))-F(y(t,x)+f(t^-))}_E=0.
\end{align}
Finally by \eqref{convFE} and the Dominated Convergence theorem we obtain \eqref{mildFFE} and so $y(\cdot,x)$ is a mild solution of \eqref{eq}, for every $x\in E$.

\textbf{Uniqueness}  
To prove the uniqueness we follow the procedure presented in \cite[Proposition 7]{Mar-Roc2010}. Let $x\in E$, $u(\cdot,x)$ and $v(\cdot,x)$ be two mild solution of \eqref{eqFO}. Let $\{e_k\}_{k\in\N}\subseteq E$ be an orthonormal basis of ${H}$. Let $\lambda\in\rho(A)$ and let $0<\eps<1$. We recall that $((\lambda\Id-\eps A)^{-1})^*=(\lambda\Id-\eps A^*)^{-1}$ and since $A$ is the infinitesimal generator of a strongly continuous semigroup, then, by the Hille-Yosida theorem $\lambda\in\rho(\eps A)$, for every $0<\eps<1$. For every $k\in\N$,  we consider the function $z:[0,T]\rightarrow \R$ defined by
{\small
\begin{align*}
z_k(t)&=\scal{u(t,x)-v(t,x)}{(\lambda\Id-\eps A^*)^{-1}e_k}_{H}\\
&=\int_0^t\scal{e^{(t-s)A}(\lambda\Id-\eps A)^{-1}(F(u(s,x)+f(s))-F(v(s,x)+f(s)))}{e_k}_{H} ds
\end{align*}
}
Noting that both Hypotheses \ref{Hyp1} and Hypotheses \ref{Hyp2} implies that for every $k\in\N$ the function $x\rightarrow\scal{F(x)}{e_k}_{\mathcal{K}}$ is continuous from $E$ to $\R$, by the same arguments used in the proof of Proposition \ref{soldt}, the function $t\rightarrow z_k(t)$ is left-differentiable for every $k\in\N$. Hence for every $t\in [0,T]$ we have
{\footnotesize
\begin{align*}
&\scal{(\lambda\Id-\eps A^*)^{-1}(u(t,x)-v(t,x))}{e_k}^2_{H}\leq \\
&\leq\int_0^t\scal{A(\lambda\Id-\eps A)^{-1}u(s,x)-v(s,x)}{e_k}_{H}\scal{(\lambda\Id-\eps A)^{-1}(u(s,x)-v(s,x))}{e_k}_{H} ds\\
&+\int_0^t\scal{(\lambda\Id-\eps A)^{-1}(F(u(s,x)+f(s))-F(v(s,x)+f(s)))}{e_k}_{H}\scal{(\lambda\Id-\eps A)^{-1}(u(s,x)-v(s,x))}{e_k}_{H} ds,
\end{align*}
}
summing up over $k\in\N$ and using the dissipativity of $A$ we obtain
{\small
\begin{align}\label{eps}
&\norm{(\lambda\Id-\eps A^*)^{-1}(u(t,x)-v(t,x))}^2\leq \zeta_A\int_0^t\norm{(\lambda\Id-\eps A^*)^{-1}(u(s,x)-v(s,x))}^2ds\\
&+\int_0^t\scal{(\lambda\Id-\eps A)^{-1}(F(u(s,x)+f(s))-F(v(s,x)+f(s)))}{(\lambda\Id-\eps A)^{-1}(u(s,x)-v(s,x))}_{H}\ ds\notag
\end{align}
}
 where $\zeta_A$ is the constant given by Hypotheses \ref{EU2}\eqref{EU2.5}. Letting $\eps\rightarrow 0$ in \eqref{eps} and using the dissipativity of $F$ we get
 \begin{align*}
\norm{u(t,x)-v(t,x)}^2&\leq \zeta\int_0^t\norm{(u(s,x)-v(s,x))}^2ds,
\end{align*}
where $\zeta$ is given by \eqref{zeta}, finally by the Gronwall inequality we obtain the uniqueness.
\end{proof}

\end{document}